\documentclass[10pt,a4paper]{article}
\usepackage[a4paper,left=2.5cm,right=2.5cm,top=2.5cm,bottom=2.5cm]{geometry}
\usepackage[british]{babel}
\usepackage[utf8]{inputenc}
\usepackage[T1]{fontenc}

\usepackage{csquotes} 
\usepackage[backend=biber, style=alphabetic, maxbibnames=4,minbibnames=4, sorting=nyt]{biblatex}
\usepackage{amsfonts}
\usepackage{amssymb}
\usepackage{amsmath}
\usepackage{tikz-cd}
\usepackage{mathtools,bm,bbm}
\usepackage[mathscr]{eucal}
\usepackage{xfrac,dsfont}
\usepackage{stackengine,units}
\usepackage{relsize}
\DeclareMathAlphabet{\mathbbmsl}{U}{bbm}{m}{sl}
\usepackage[colorlinks=true, allcolors=blue]{hyperref}
\usepackage{enumitem}
\usepackage{calc}

\usepackage{amsthm}
\newtheorem{theorem}{Theorem}[section]
\newtheorem{lemma}[theorem]{Lemma}

\newtheorem{definition}[theorem]{Definition}
\newtheorem{example}[theorem]{Example}
\newtheorem{corollary}[theorem]{Corollary}
\newtheorem{remark}[theorem]{Remark}
\newtheorem*{theorem*}{Theorem}

\newcommand{\N}{\mathbb{N}}
\newcommand{\Z}{\mathbb{Z}}
\newcommand{\R}{\mathbb{R}}
\newcommand{\C}{\mathbb{C}}
\newcommand{\frA}{\mathfrak{A}}
\newcommand{\frg}{\mathfrak{g}}
\newcommand{\frh}{\mathfrak{h}}
\newcommand{\frt}{\mathfrak{t}}
\newcommand{\frk}{\mathfrak{k}}

\newcommand{\frp}{\mathfrak{p}}

\newcommand{\Dcal}{\mathcal{D}}
\newcommand{\Hcal}{\mathcal{H}}
\newcommand{\Lcal}{\mathcal{L}}

\newcommand{\Ecal}{\mathcal{E}}

\newcommand{\dolb}{\bar{\partial}}

\DeclareMathOperator{\Tr}{Tr}
\DeclareMathOperator{\image}{Im}
\DeclareMathOperator{\kernel}{Ker}

\DeclareMathOperator{\Sp}{Sp}

\DeclareMathOperator{\SU}{SU}
\DeclareMathOperator{\End}{End}

\DeclareMathOperator{\Id}{Id}
\DeclareMathOperator{\Ad}{Ad}
\DeclareMathOperator{\Aut}{Aut}
\DeclareMathOperator{\proj}{proj}
\DeclareMathOperator{\Hom}{Hom}
\DeclareMathOperator{\prim}{prim}
\DeclareMathOperator{\ad}{ad}

\newcommand{\spn}[1]{\Sp({#1})}

\newcommand{\metric}[2]{\langle #1,#2 \rangle}

\begin{document}
\title{Analytic Torsion for Symmetric Contact Manifolds}
\author{Niklas Henningsen\thanks{Mathematisch-Naturwissenschaftliche Fakult\"at, Mathematisches Institut, Heinrich-Heine-Universit\"at D\"usseldorf, Universit\"atsstraße
1, 40225 D\"usseldorf, Germany\\ Email: \texttt{niklas.henningsen@hhu.de}}}
\date{07 September 2026}
\maketitle

\begin{abstract}        
  \noindent We review the definitions of the Rumin complex and analytic contact torsion for contact manifolds. We then compute the contact torsion for symmetric contact manifolds. To this end, we use a representation-theoretic description of the relevant operators and spaces of differential forms. Our computation of the contact torsion for symmetric contact manifolds relies on a recent result by Rumin that reduces the problem to harmonic forms. Following a construction by Boothby and Wang, the contact manifolds under consideration are total spaces of $S^1$\nobreakdash-principal bundles over K\"ahler manifolds. This allows us to interpret the harmonic forms as Dolbeault cohomology with coefficients in a certain holomorphic line bundle of the base K\"ahler manifold, thereby simplifying the computation by allowing us to use results from complex geometry. Furthermore, we generalise Rumin's result to the equivariant case and, using the previous considerations, determine the equivariant contact torsion for the case of isolated fixed points.
\end{abstract}


\section{Introduction}
\label{contactsect5}
Contact geometry can be understood as the odd dimensional analogue of symplectic geometry, sharing many of its algebraic structures. While a symplectic manifold of dimension $2n$ possesses a closed and non-degenerate bilinear form, a contact manifold $M$ of dimension $2n+1$ is equipped with a maximally non-integrable hyperplane distribution $\Dcal\subset TM$, which is induced by local $1$-forms $\theta$ such that $d\theta$ is non-degenerate on $\Dcal$.

\noindent Therefore, a contact manifold has an additional distinct direction within its tangent bundle which is pointing outside the contact structure $\Dcal$. For a contact manifold with globally defined contact form, this direction is canonically generated by a unique vector field, known as the Reeb vector field $T$.

\noindent Despite the differences and the influence of the additional direction, there are nevertheless some interesting similarities with K\"ahler geometry if the contact structure can be equipped with a suitable complex structure. These similarities will become very helpful as we proceed.

In \cite{ruminformesdiffer}, Rumin defined a complex adapted to contact geometry, which represents an analogue of the de Rham complex specifically for contact geometry. Essentially, this so-called Rumin complex arises by dividing out an ideal from the differential forms of $M$, which is generated by the contact form $\theta$ and its exterior derivative $d\theta$. A distinctive feature is the middle degree, since in degree $n$ the differential operator of the Rumin complex is a differential operator of order $2$, whilst in all other degrees they are differential operators of order $1$. In \cite[Equations~(34)]{ruminsubriemannian}, Rumin showed that for contact manifolds whose complex structure is integrable with respect to the contact structure, such that $\Lcal_T(J)=0$, re\-scaling the operators in the Rumin complex implies that the resulting operators satisfy K\"ahler-type identities, which prove to be very helpful in other applications.

Analytic torsion (also known as Ray-Singer torsion) was defined in \cite{RaySingerAnalyticTorsion} as an analytic analogue of Reidemeister torsion \cite{Reidemeister}.
Unlike the Hodge Laplacian, the Rumin Laplacian as defined with the rescaled \mbox{Rumin} complex is not elliptic, but hypoelliptic on compact contact manifolds. This made it possible, in \cite{rumin2012analytic}, to define an analytic contact torsion based on the Rumin Laplacian~$\Delta_\Ecal$.
Rumin and Seshadri defined the zeta function
\begin{align*}
    Z(s) := \frac{1}{2} \sum\limits_{k=0}^{2n+1} (-1)^k \omega(k) \zeta(\Delta_\Ecal\vert \Ecal^k(M))(s),
\end{align*}
where
\begin{align*}
    \zeta(\Delta_\Ecal\vert \Ecal^k(M)) (s) := \dim (H_\Ecal^k(M)) + \sum\limits_{\lambda\in\operatorname{spec}^\ast (\Delta_\Ecal\vert \Ecal^k(M))} \frac{1}{\lambda^s},
\end{align*}
and $H_\Ecal^k(M)$ denotes the cohomology of the Rumin complex. The analytic contact torsion is defined as $T_\Ecal(M):=-\frac{1}{2}Z'(0)$, where $Z'(0)$ is evaluated using the meromorphic continuation of the zeta function $Z(s)$, which is holomorphic in $s=0$.
This definition differs only slightly from the original definition of contact torsion, $T_\Ecal(M):=\exp(-\frac{1}{2}Z'(0))$, as the exponential function is omitted here. Building on this, Teßmer defined the equivariant contact torsion in \cite{tessmer}.

In \cite{rumin2012analytic}, it was shown for $3$-dimensional CR Seifert manifolds that, for the twisted version with a flat unitary vector bundle, the contact torsion (based on the unscaled Rumin Laplacian) coincides with the Ray-Singer torsion. 
A similar result for spheres from \cite{kitaokaspheres} states that
\begin{align*}
    T_\Ecal (S^{2n+1}) = n!T_\mathrm{RS}(S^{2n+1}),
\end{align*}
where $T_\mathrm{RS}(S^{2n+1})$ denotes the Ray-Singer torsion and $T_\Ecal (S^{2n+1})$ denotes the contact torsion (as defined by Rumin with the exponential function) based on the rescaled Rumin Laplacian. In \cite[Corollary 3]{AlbinQuanSubRiemannian}, Albin and Quan showed that the Ray-Singer torsion and the contact torsion based on the unscaled Rumin Laplacian differ by an integral of a local invariant of the metric. They also showed in \cite[Equation~(7.1)]{AlbinQuanSubRiemannian} that the unscaled and rescaled contact torsion differ by a local term.

In \cite{rumin2026}, Rumin showed for CR Seifert manifolds $M$, i.e.\ compact contact manifolds with integrable complex structure on $\Dcal$ and $\Lcal_T(J)=0$ such that $M$ admits a circle action generated by the Reeb vector field $T$, that the heat torsion function can be expressed entirely with the space of harmonic forms of $\Delta_{\dolb_\Dcal}=\dolb_\Dcal\dolb_\Dcal^\ast + \dolb_\Dcal^\ast\dolb_\Dcal$, on which $\Delta_\Ecal = -\Lcal_T^2$ holds. Here, $d_\Dcal = \partial_\Dcal + \dolb_\Dcal$ is the de Rham operator on the differential forms of the contact structure $\frA^\bullet (\Dcal)$ followed by a projection onto the $\frA^\bullet (\Dcal)$. In \cite{KitaokaHarmonicSasakian}, Kitaoka proved a similar result for Sasakian manifolds, which, through several cancellations of terms, states that the zeta function can be reduced to the eigenspaces of $-\Lcal_T^2$ for certain other harmonic forms. The harmonic forms relevant to \cite{KitaokaHarmonicSasakian} were studied in \cite{GarfieldLeeRuminComplex} and \cite{CaseRuminComplex}. Further information and results on Sasakian manifolds can be found in \cite{BoyerGalickiSasakianManifolds}.
In \cite{kitaokaspheres}, Kitaoka determined the contact torsion on the basis of the rescaled Rumin Laplacian of $S^{2n+1}\subset \C^{n+1}$ using the representation theory of $\operatorname{U}(n+1)$. The relevant representation theory of $\operatorname{U}(n+1)$ was investigated in more detail in \cite{JulgKasparov} for the complexification of the spaces $\Ecal^k(M)$ of the Rumin complex, by decomposing these spaces into irreducible representations. The irreducible representations occur with multiplicity at most $1$, which enabled the explicit calculation of the eigenvalues of the rescaled Rumin Laplacian. In a subsequent article \cite{kitaokalensspaces}, Kitaoka used the methods already employed in \cite{kitaokaspheres} to also compute the contact torsion for lens spaces.

In \cite{contactmanifoldsboothby}, Boothby and Wang examined so-called regular contact manifolds and showed in \cite[Theorem~1]{contactmanifoldsboothby} that there exists a rescaled contact form $\theta$ which induces an action of $S^1$ on the contact manifold $M$. With respect to this rescaled contact form, $M$ is the total space of a $S^1$\nobreakdash-principal bundle over a symplectic manifold $N$, where the pullback of the symplectic form of $N$ yields exactly $d\theta$. In \cite[Theorem~3]{contactmanifoldsboothby}, they also provided a converse to this theorem. In \cite{contactmanifoldsNURboothby}, building on \cite{contactmanifoldsboothby}, Boothby showed that for every compact homogeneous contact manifold $G/H$, i.e.\ a homogeneous manifold where the contact form is $G$-invariant, one obtains a K\"ahler manifold $G/K$ with $H\subset K$ such that $G/H\rightarrow G/K$ is a $S^1$-principal bundle. Fibre bundles with fibre $S^1$ are generally well-studied, for example in the context of analytic torsion forms, see, for example, \cite{BismutLott}.\\

The (homogeneous) symmetric contact manifolds $G/H$ for semisimple Lie groups $G$ were classified by Alekseevsky and Gorodski in \cite{semisimplecontactalekseevsky} by using the results of \cite{BieliavskySymplectic}. The explicit reference to homogeneity is necessary here, as in \cite{semisimplecontactalekseevsky} they did not use the classical definition of symmetric spaces, instead one considers a definition of symmetry adapted to contact manifolds, which requires the symmetry property only on the contact structure. As a result, there are also examples of non-homogeneous symmetric contact manifolds, see \cite{Zalabova}. However, as in \cite{semisimplecontactalekseevsky}, we consider only the homogeneous cases. The compact CR symmetric contact manifolds relevant to our purposes occur, as described above, as the total space of a $S^1$-principal bundle $G/H\rightarrow G/K$ over a symmetric K\"ahler manifold $G/K$, and they possess an action of $S^1$ induced by the Reeb vector field. The spheres $S^{2n+1}$ are classic examples of such spaces. Motivated by the computation of the contact torsion in \cite{kitaokaspheres} for spheres, our aim is to compute the contact torsion for general compact CR symmetric contact manifolds.

\noindent To this end, one could first attempt, following the approach in \cite{ikedataniguchiLaplace}, to express the Rumin Laplacian on the induced $G$-representations entirely in terms of representation theory, in the same way as the Hodge Laplacian. The results from \cite{ikedataniguchiLaplace} for the Hodge Laplacian were, for example, used in \cite{kkhermsymm} to compute the holomorphic torsion of a Hermitian symmetric space. However, this approach proves to be very difficult in the context of contact manifolds and the associated Rumin complex, as the calculation of eigenvalues is significantly complicated by the presence of projections onto subspaces in the definition of the differential operators and by the differential operator in middle degree $n$. The approach from \cite{rumin2026} is therefore a welcome alternative: thanks to Rumin's results, the computation of the contact torsion reduces to the calculation of the Lie derivative along the Reeb vector field, i.e.\ precisely the direction that is orthogonal to the contact structure $\Dcal$.\\

In Section~\ref{contactsect6}, we first revisit the definition of a contact manifold and the construction of the Rumin complex. The definition of analytic contact torsion follows in Section~\ref{contactsect7}. In Section~\ref{contactsect8}, we begin by describing, following Ikeda and Taniguchi \cite{ikedataniguchiLaplace}, the relevant operators and spaces of differential forms using methods from representation theory. We consider a decomposition of the relevant spaces with respect to the $S^1$-action. We then show, on the weight spaces of this $S^1$-action, that $-\Lcal_T^2$ acts as $\left(\frac{m+p-q}{2n}\right)^2$ on the $(p,q)$-forms with $S^1$-weight $m$. In Section~\ref{contactsect8.2}, we show, with the help of the fibre bundle $G/H\rightarrow G/K$, that the harmonic $(p,q)$-forms of $\Delta_{\dolb_\Dcal}$ of $S^1$-weight $m$ are isomorphic to the Dolbeault cohomology $H^{p,q}_{\dolb} (G/K,\Lcal^{m+p-q})$ of $G/K$ for a certain holomorphic associated line bundle $\Lcal^{m+p-q}\rightarrow G/K$.
Together with Rumin's result described above \cite{rumin2026}, we then compute the contact torsion for compact CR symmetric spaces in Section~\ref{contactsect8.3}. The contact torsion is given by
\begin{align*}
    T_\Ecal (G/H) = {\frac{\vert W_G\vert}{\vert W_K\vert}} \cdot \ln\left(4\pi \frac{n}{N}\right),
\end{align*}
where $W_G$ and $W_K$ are the Weyl groups of $G$ and $K$ respectively, and $N = \vert \Gamma \vert$ is the order of $\Gamma$ for $K = (H \times S^1)/\Gamma$. We also derive a product formula for the contact torsion of compact CR symmetric contact manifolds.\\
In Section~\ref{contactsect8.4}, we generalise Rumin's result for the equivariant case, which was introduced in \cite{tessmer}, with respect to the action of elements of $G$. For the action of an element $t=\exp_G(X)$ of the maximal torus, which generates the maximal torus, the fixed points are isolated. Thus, due to the Atiyah-Bott fixed point theorem \cite{atiyahbott2}, we obtain the equivariant contact torsion
\begin{align*}
    T_\Ecal (M,t) =& T_\Ecal (M) + \frac{\vert W_G\vert}{\vert W_K\vert}\gamma + \sum\limits_{[w]\in W_G/W_K} \frac{1}{2}(\psi([(w\cdot \lambda)(X)]) + \psi([-(w\cdot \lambda)(X)])),
\end{align*}
where $\psi(x)=\frac{\Gamma'(x)}{\Gamma(x)}$ is the Digamma function, $\gamma$ is Euler's constant and $[(w\cdot \lambda)(X)]$, $[-(w\cdot \lambda)(X)]$ are the projections to $]0,1[$ for $(w\cdot \lambda)(X)\in\R\setminus\Z$, i.e.\ the unique representatives in $]0,1[$ modulo $1$.\\

An interesting starting point for further research would be to investigate to what extent the methods used can be generalised to the case of quotients of CR symmetric contact manifolds by the action of finite cyclic subgroups that leave the contact structure invariant, such as in the case of lens spaces. One challenge here is that, in this case, the base manifold is generally not a manifold but an orbifold. For example, $\mathbf{P}^n\C$, which appears in the context of spheres $S^{2n+1}$, is no longer a manifold when divided by some of the groups that define the lens spaces. Many classical results are also available in the orbifold setting, such as the Hirzebruch-Riemann-Roch theorem for orbifolds, see \cite{Kawasaki}.\\

This work is the shortened version of a part of the author’s PhD thesis that is currently under review.
\section{Contact Manifolds}
\label{contactsect6}
\subsection{Preliminaries}
\label{contactsect6.1}
In this section, we first review the basic theory of contact manifolds, following \cite{ruminformesdiffer}, \cite{ruminsubriemannian} and \cite{rumin2012analytic}.
\begin{definition}
    A \textbf{contact manifold} is a pair $(M,\Dcal)$ consisting of a $2n+1$-dimensional manifold $M$ and a hyperplane distribution $\Dcal\subset TM$ such that for every $x\in M$, there exists an open neighbourhood $U\subset M$ with $x\in U$ and a $1$-form $\theta\in \frA^1(U)$ with
    \begin{enumerate}
        \item $\kernel (\theta) = \Dcal_{\vert U}$,
        \item $\theta\wedge (d\theta)^n_{\vert x}\neq 0$ for all $x\in U$.
    \end{enumerate}
    The hyperplane distribution $\Dcal$ is called \textbf{contact structure} and the $1$-form $\theta$ is called a (local) \textbf{contact form}.
\end{definition}
From now on, we will focus exclusively on compact contact manifolds that possess a \textbf{globally defined contact form}. This means that all contact manifolds under consideration are always orientable, due to the requirement that $\theta\wedge (d\theta)^n_{\vert x}\neq 0$ for all $x\in M$.
\begin{definition}
    For every contact manifold $M$ with contact form $\theta\in\frA^1(M)$ there exists a unique vector field $T\in\Gamma(M,TM)$, called the \textbf{Reeb vector field}, such that
    \begin{enumerate}
        \item $\theta(T)=1$,
        \item $d\theta (T,\cdot) = 0$.
    \end{enumerate}
\end{definition}
According to \cite{McDuffSalamonSymplectic} and \cite{tessmer}, we can always choose an almost complex structure $J\in\Gamma (M,\End(\Dcal))$ such that $J^2=-\Id_\Dcal$, $d\theta(J\cdot,J\cdot)=d\theta$ and $d\theta(\cdot,J\cdot)$ is a positive definit metric on $\Dcal$. Let $J(T):=0$. We then proceed to define
\begin{align*}
    g:=\theta\otimes \theta + d\theta(\cdot,J\cdot)
\end{align*}
as the metric on $M$. With regard to this metric, we have $\Dcal\perp \R\cdot T$.

\noindent Since we are working with compact contact manifolds, we obtain the $L^2$-scalar product on $\Gamma (M,TM)$ (respectively $\frA^\bullet (M)$)
\begin{align*}
    \langle X,Y\rangle_{L^2} := \int_M g(X,Y)\operatorname{dvol}_g,
\end{align*}
where $\operatorname{dvol}_g$ is the volume form defined by the metric $g$. It follows that $T^\flat:=g(T,\cdot)=\theta$ and $(\theta)^\ast = \iota_T$. Euclidean/Riemannian metrics are canonically extended to Hermitian metrics where necessary.\\
Since $TM=\R\cdot T\stackrel{\perp}{\oplus} \Dcal$ and, accordingly, $T^\ast M=\R\cdot \theta\stackrel{\perp}{\oplus} \Dcal^\ast$, we can now decompose the exterior powers of $T^\ast M$ orthogonally with respect to this decomposition:
\begin{align*}
    \Lambda^k T^\ast M = \R\cdot \theta\wedge\Lambda^{k-1}\Dcal^\ast \stackrel{\perp}{\oplus} \Lambda^k\Dcal^\ast.
\end{align*}
This, together with $\frA^k(\Dcal):=\Gamma (M,\Lambda^k \Dcal^\ast)=\kernel ({\iota_T}_{\vert \frA^k(M)})$, provides a decomposition of the differential forms
\begin{align}
    \frA^k(M)=\theta\wedge\frA^{k-1}(\Dcal)\stackrel{\perp}{\oplus} \frA^k(\Dcal). \label{contactlabel1}
\end{align}
We will consider some basic results regarding symplectic structures from \cite[Chap\-ter~I]{WeilKaehler}, \cite{McDuffSalamonSymplectic}, \cite{griffiths1994principles} and \cite{tessmer}.\\
There always exists a local basis of $\Dcal$ of the form
\begin{align}
    X_1,\ldots,X_n,X_{n+1}=J(X_1),\ldots,X_{2n}=J(X_n).\label{contactlabel27}
\end{align}
This also provides us with a local (complex) basis 
\begin{align*}
    (X_j^{1,0}:=X_j-iX_{n+j})_{1\leq j\leq n} &\text{ of } \Dcal^{1,0}:=\kernel ((J-i\cdot\Id)_{\vert \Dcal\otimes_\R\C})\\
    \text{and } (X_j^{0,1}:=X_j+iX_{n+j})_{1\leq j\leq n} &\text{ of } \Dcal^{0,1}:=\kernel ((J+i\cdot\Id)_{\vert \Dcal\otimes_\R\C}).
\end{align*}
Let $\theta^j := g(X_j,\cdot)$ for $1 \leq j \leq 2n$, then
\begin{align}
    d\theta = \sum\limits_{j=1}^n \theta^j\wedge\theta^{n+j}. \label{contactlabel2}
\end{align}
Thus, $d\theta$ is a symplectic form on $\Dcal$.
\begin{definition}
    The \textbf{Lefschetz operator} is defined as
    \begin{align*}
        L\colon \frA^k(\Dcal) &\rightarrow \frA^{k+2}(\Dcal)\\
        \alpha &\mapsto (d\theta)\wedge \alpha,
    \end{align*}
    and let
    \begin{align*}
        \Lambda:=L^\ast\colon \frA^k(\Dcal) \rightarrow \frA^{k-2}(\Dcal)
    \end{align*}
    be its adjoint operator.
\end{definition}
Let 
\begin{align*}
    \Lambda^{p,q}\Dcal :=& \Lambda^p (\Dcal^{1,0})^\ast \wedge \Lambda^q (\Dcal^{0,1})^\ast,\\
    \frA^{p,q}(\Dcal) :=& \Gamma (M,\Lambda^{p,q}\Dcal),\\
    \frA^k_\C(\Dcal) :=& \bigoplus\limits_{p+q=k}\frA^{p,q} (\Dcal). 
\end{align*}
Due to equation (\ref{contactlabel2}) and $d\theta\in\frA^{1,1}(\Dcal)$, we have
\begin{align}
    L\colon \frA^{p,q}(\Dcal) &\rightarrow \frA^{p+1,q+1}(\Dcal) \label{contactlabel17}
\end{align}
and
\begin{align*}
    \Lambda\colon \frA^{p,q}(\Dcal) \rightarrow \frA^{p-1,q-1}(\Dcal).
\end{align*}
The Lefschetz operator provides us with a decomposition of $\frA^\bullet(\Dcal)$. The space of \textbf{primitive forms} is defined as $\frA_{\prim}^\bullet(\Dcal):=\frA_{\prim}^\bullet(\Dcal)\cap \kernel (\Lambda)$.
\begin{lemma}
\label{contactlabel10}
    \begin{enumerate}
        \item There exists an orthogonal decomposition known as the \mbox{Lefschetz} decomposition:
        \begin{align}
            \frA^k (\Dcal) &= \bigoplus_{j\geq 0} L^j (\frA^{k-2j}_{\prim}(\Dcal)),\label{contactlabel9}\\
            \frA^{p,q} (\Dcal) &= \bigoplus_{j\geq 0} L^j (\frA^{p-j,q-j}_{\prim}(\Dcal)).\notag
        \end{align}
        \item For $k\geq n+1$, we have $\frA^k_{\prim} (\Dcal) = 0$.
        \item For $k\leq n$, $L^{n-k}\colon \frA^k_{\prim} (\Dcal)\rightarrow \frA^{2n-k} (\Dcal)$ is injective and $L^{n-k}\colon \frA^k (\Dcal)\rightarrow \frA^{2n-k} (\Dcal)$ is bijective.
        \item For $k\leq n$, $\frA^k_{\prim} (\Dcal)= \lbrace \alpha\in \frA^k (\Dcal)\mid L^{n-k+1}\alpha=0\rbrace$.
    \end{enumerate}
\end{lemma}
\begin{remark}
    The statement in Lemma~\hyperref[contactlabel10]{\ref*{contactlabel10}.(4)} follows from the formula $[L^j,\Lambda]=j(k-n+j-1)L^{j-1}$ on $\frA^k_{\prim}(\Dcal)$ from \cite[p.~22]{WeilKaehler} and due to Lemma~\hyperref[contactlabel10]{\ref*{contactlabel10}.(3)}. Further information can also be found in \cite[Chapter~V]{WellsComplexManifolds}.
\end{remark}
\begin{example}[\cite{kitaokaspheres}, \cite{kitaokalensspaces}]
    \begin{enumerate}
        \item The sphere $S^{2n+1}:=\lbrace z\in\C^{n+1}\mid \lVert z \rVert = 1\rbrace$ is a compact contact manifold equipped with the contact form
    \begin{align*}
        \theta := i(\dolb - \partial)\lVert z \rVert^2
    \end{align*}
    and the associated Reeb vector field
    \begin{align*}
        T=\frac{i}{2}\sum\limits_{j=1}^{n+1} (z_j\frac{\partial}{\partial z_j}-\overline{z}_j\frac{\partial}{\partial \overline{z}_j}).
    \end{align*}
    The almost complex structure on $\Dcal$ is given by the standard fast-complex structure on $\C^{n+1}$.
    \item For integers $\mu,\nu_1,\ldots,\nu_{n+1}$ such that the $\nu_j$ are coprime to $\mu$, let $\Gamma$ be the subgroup of $(S^1)^{n+1}$ generated by
    \begin{align*}
        \gamma = (\gamma_1,\ldots,\gamma_{n+1}) := \left(\exp \left(2\pi i \frac{\nu_1}{\mu}\right),\ldots,\exp \left(2\pi i \frac{\nu_{n+1}}{\mu}\right)\right).
    \end{align*}
    The corresponding lens space is defined by
    \begin{align*}
        S^{2n+1}/\Gamma.
    \end{align*}
    The lens spaces are contact manifolds that inherit their contact structure from $S^{2n+1}$, since the sphere's contact structure is invariant under the action of $\Gamma$.
    \end{enumerate} 
\end{example}
\noindent Spheres are examples of homogeneous contact manifolds, which we will examine in more detail in a later section.
\subsection{The Rumin Complex}
\label{contactsect6.2}
The Rumin complex serves as a contact geometric analogue of the classical de Rham complex. By modifying the de Rham complex, essentially by dividing out the ideal generated by the contact form and its exterior derivative, it yields a complex specifically tailored to the study of the underlying contact structure of M. In the following section, we summarise its construction as introduced by Rumin in \cite{ruminformesdiffer} and \cite{ruminsubriemannian}.\\
From now on, we restrict our attention to compact contact manifolds of dimension $2n+1$ equipped with a globally defined contact form $\theta$. As mentioned already, the global existence of $\theta$ ensures that $M$ is orientable.\\

Let $\proj_{\frA^\bullet (\Dcal)}$ be the orthogonal projection onto $\frA^\bullet (\Dcal)$. We define
\begin{align*}
    d_\Dcal\colon \frA^k (\Dcal) &\rightarrow \frA^{k+1} (\Dcal)\\
    \alpha &\mapsto \proj_{\frA^\bullet (\Dcal)}(d\alpha),
\end{align*}
and due to the Cartan homotopy formula $\Lcal_T=d\circ \iota_T + \iota_T\circ d$ \cite[Theorem~2.3.7]{kkdiffgeoENG}, one gets
\begin{align*}
    d_\Dcal = d-(\theta\wedge)\circ \Lcal_T,
\end{align*}
since $\frA^\bullet (\Dcal)=\kernel (\iota_T)$. 
We also have
\begin{align*}
    \Lcal_T&\colon \frA^\bullet (\Dcal)\rightarrow \frA^\bullet (\Dcal),\\
    \Lcal_T&\colon \theta\wedge\frA^\bullet (\Dcal)\rightarrow \theta\wedge\frA^\bullet (\Dcal),
\end{align*}
because of $\Lcal_T(\alpha)=\iota_T(d\alpha)$ and $\Lcal_T(\theta\wedge\alpha) = \Lcal_T(\theta)\wedge \alpha + \theta\wedge\Lcal_T(\alpha) = \theta\wedge\Lcal_T(\alpha)$ due to $\Lcal_T(\theta)=(d\circ \iota_T+\iota_T\circ d)\theta=0$ for $\alpha\in\frA^\bullet(\Dcal)$.\\
For $\theta\wedge\alpha + \beta\in \theta\wedge\frA^{k-1}(\Dcal)\oplus \frA^k(\Dcal)$, the de Rham operator can now be decomposed with respect to the decomposition (\ref{contactlabel1})
\begin{align}
    d(\theta\wedge\alpha + \beta) = \theta\wedge (-d_\Dcal \alpha + \Lcal_T\beta) + ((d\theta)\wedge \alpha + d_\Dcal \beta). \label{contactlabel3}
\end{align}
Since $d^2 = 0$, the following operator equations on $\frA^\bullet (\Dcal)$ follow directly from (\ref{contactlabel3})
\begin{align}
    d_\Dcal^2&=-L\Lcal_T=-\Lcal_T L,\label{contactlabel15}\\
    0&=[d_\Dcal,L]=[d_\Dcal,\Lcal_T]=[L,\Lcal_T].\label{contactlabel16}
\end{align}
If, in addition, $\Lcal_T^\ast=-\Lcal_T$, then by taking the adjoint, we have
\begin{align*}
    0&=[d_\Dcal^\ast,\Lcal_T]=[\Lambda,\Lcal_T].
\end{align*}
Thus, it is easy to see that $\Lcal_T$ is compatible with the Lefschetz decomposition.\\
We also have $[\Lambda,d_\Dcal]=-J^{-1}d_\Dcal^\ast J$, where $(J\alpha)(Y_1,\ldots,Y_k) := \alpha (JY_1,\ldots,JY_k)$ for $\alpha\in \frA^k(\Dcal)$.
\begin{definition}
    An almost complex structure $J\in\Gamma(M,\End(\Dcal))$ is called \textbf{integrable} if
    \begin{align*}
        [\Gamma(M,\Dcal^{1,0}),\Gamma(M,\Dcal^{1,0})]\subset \Gamma(M,\Dcal^{1,0}).
    \end{align*}
\end{definition}
\noindent In the case of an integrable complex structure, we have
\begin{align*}
    d_\Dcal = \partial_\Dcal + \dolb_\Dcal
\end{align*}
on $\frA^{p,q}(\Dcal)$ with $\partial_\Dcal\colon \frA^{p,q}(\Dcal)\rightarrow \frA^{p+1,q}(\Dcal)$ and $\dolb_\Dcal\colon \frA^{p,q}(\Dcal)\rightarrow \frA^{p,q+1}(\Dcal)$.\\
\hfill\break
Since $d_\Dcal^2\neq 0$, this operator does not provide a complex. In order to obtain a complex, Rumin defined the following spaces:
\begin{align*}
    \Ecal^k (M) := \begin{cases}
        \frA^k_{\prim}(\Dcal) \text{ for } k\leq n,\\
        \theta\wedge(\frA^k(\Dcal)\cap\kernel(L)) \text{ for } k\geq n+1.
    \end{cases} 
\end{align*}
\begin{remark}
    We have $\frA^k_{\prim}(\Dcal)=0$ for $k\geq n+1$ and $\theta\wedge(\frA^k(\Dcal)\cap\kernel(L))=0$ for $k\leq n$, see Lemma~\ref{contactlabel10}.
\end{remark}
\noindent We can now define suitable differential operators on the spaces above, such that they form a complex:
\begin{itemize}
    \item For $k\leq n-1$, let $d_\mathrm{R}:=\proj_{\prim}\circ d_\Dcal\colon \frA^k_{\prim}(\Dcal)\rightarrow\frA^{k+1}_{\prim}(\Dcal)$, where $\proj_{\prim}$ is the orthogonal projection onto the primitive forms.
    \item For $k=n$, let $D:=(\theta\wedge)\circ (\Lcal_T+d_\Dcal L^{-1}d_\Dcal)\colon \frA^n_{\prim}(\Dcal)\rightarrow \theta\wedge(\frA^{n+1}(\Dcal)\cap\kernel(L))$.
    \item For $k\geq n+1$, let $d_\mathrm{R}:=d_{\vert \theta\wedge(\frA^k(\Dcal)\cap\kernel(L))}\colon \theta\wedge(\frA^k(\Dcal)\cap\kernel(L))\rightarrow \theta\wedge(\frA^{k+1}(\Dcal)\cap\kernel(L))$, which is well-defined because of $0=L\alpha=(d\theta)\wedge \alpha$ for $\theta\wedge\alpha\in \theta\wedge(\frA^k(\Dcal)\cap\kernel(L))$.
\end{itemize}
These operators provide a complex of differential operators, known as the \textbf{Rumin complex}:
\begin{align*}
    \mathcal{E}^0(M)\xrightarrow{d_\mathrm{R}}\mathcal{E}^1(M)\xrightarrow{d_\mathrm{R}}\ldots\xrightarrow{d_\mathrm{R}}\mathcal{E}^n(M)\xrightarrow{D}\mathcal{E}^{n+1}(M)\xrightarrow{d_\mathrm{R}}\ldots\xrightarrow{d_\mathrm{R}}\mathcal{E}^{2n+1}(M).
\end{align*}
\begin{remark}
        There is another equivalent definition of the Rumin complex, in which the spaces of degrees $0\leq k\leq n$ are replaced by $\frA^k(M)/\mathcal{I}^k$ with the ideal $\mathcal{I}^\bullet\subset \frA^\bullet(M)$, generated by $\theta$ and $d\theta$, as mentioned at the beginning of the section. Every equivalence class $[\alpha]$ in $\frA^k(M)/\mathcal{I}^k$ has a unique primitive representative $\alpha\in\frA^k_{\prim}(\Dcal)$. This provides us with the required isomorphism. Further details can be found in \cite{ruminformesdiffer} or \cite{tessmer}.
\end{remark}
\noindent However, we will be working with a rescaled version of this complex.
\begin{definition}
    The \textbf{(rescaled) Rumin complex} is defined by
    \begin{align*}
        \mathcal{E}^0(M)\xrightarrow{d_\Ecal}\mathcal{E}^1(M)\xrightarrow{d_\Ecal}\ldots\xrightarrow{d_\Ecal}\mathcal{E}^n(M)\xrightarrow{D}\mathcal{E}^{n+1}(M)\xrightarrow{d_\Ecal}\ldots\xrightarrow{d_\Ecal}\mathcal{E}^{2n+1}(M),
    \end{align*}
    where $d_\Ecal:=\frac{1}{\sqrt{\vert n-k\vert}} d_\mathrm{R}$ on $\Ecal^k(M), k\neq n$. The cohomology of this complex is denoted by $H^\bullet_\Ecal(M)$.\\
    We define the \textbf{Rumin Laplacian} by
    \begin{align*}
        \Delta_\Ecal := \begin{cases}
            (d_\Ecal d_\Ecal^\ast + d_\Ecal^\ast d_\Ecal)^2\colon \Ecal^k(M)\rightarrow \Ecal^k(M) \text{ for } k\not\in {n,n+1},\\
            D^\ast D + (d_\Ecal d_\Ecal^\ast)^2\colon \Ecal^n(M)\rightarrow \Ecal^n(M),\\
            (d_\Ecal^\ast d_\Ecal)^2 + DD^\ast\colon \Ecal^{n+1}(M)\rightarrow \Ecal^{n+1}(M).
        \end{cases}
    \end{align*} 
\end{definition}
\begin{lemma}[{\cite[pp.~286-290]{ruminformesdiffer}}]
\label{contactlabel7}
    \begin{enumerate}
        \item We have $\kernel (\Delta_\Ecal)\cap \Ecal^k(M)\cong H^k_\Ecal(M)$ and the cohomology of the Rumin complex is isomorphic to the cohomology of the de Rham complex of $M$.
        \item The Hodge $\ast$ operator is an isomorphism $\ast\colon \Ecal^k(M)\rightarrow \Ecal^{2n+1-k}(M)$.
        \item For $k\neq n+1$, we have $d_\Ecal^\ast = (-1)^k\ast d_\Ecal \ast$ on $\Ecal^k(M)$, and on $\Ecal^n(M)$ we have $D^\ast = (-1)^{n+1} \ast D\ast$.
        \item We have $\Delta_\Ecal\ast=\ast\Delta_\Ecal$.
    \end{enumerate}
\end{lemma}
\begin{remark}
    In \cite{rumin2012analytic}, it is stated that $d_\mathrm{R}^\ast = (-1)^k\ast d_\mathrm{R} \ast$ on $\Ecal^k(M)$; however, since $a_{k-1}=a_{2n+1-k}$, this implies  the equation above for the rescaled operator.
\end{remark}
The following lemma shows why it is useful to work with the rescaled Rumin complex. For an integrable complex structure $J$, note that ${d_\Ecal}_{\vert \Ecal^k(M)}$ for $k<n$ can also be decomposed, in the same way as $d_\Dcal$, into $d_\Ecal = \partial_\Ecal + \dolb_\Ecal$, since the Lefschetz decomposition is compatible with the bigradation of $\frA^{\bullet,\bullet} (\Dcal)$.\\
This allows us to define
\begin{align*}
    \Delta_{\partial_\Ecal} := \partial_\Ecal\partial_\Ecal^\ast + \partial_\Ecal^\ast\partial_\Ecal, && \Delta_{\dolb_\Ecal} := \dolb_\Ecal\dolb_\Ecal^\ast + \dolb_\Ecal^\ast\dolb_\Ecal.
\end{align*}
\begin{definition}
    A contact manifold $M$ with contact form $\theta$ and an integrable complex structure $J$ on the contact structure $\Dcal$ such that $\Lcal_T(J)=0$ is called a \textbf{Sasakian manifold}.
\end{definition}
\begin{lemma}[\cite{ruminsubriemannian}, \cite{rumin2026}]
\label{contactlabel8}
    If the almost complex structure $J$ is integrable on $\Dcal$ and $\Lcal_T (J)=0$, i.e.\ $M$ is Sasakian, then the following equations hold:
    \begin{enumerate}
        \item $0=\partial_\Ecal^2=\dolb_\Ecal^2=\partial_\Ecal\dolb_\Ecal + \dolb_\Ecal\partial_\Ecal$ on $\Ecal^k(M)$ for $k\leq n-2$,
        \item $0=\partial_\Ecal^\ast\dolb_\Ecal + \dolb_\Ecal^\ast\partial_\Ecal = \dolb_\Ecal^\ast\partial_\Ecal + \partial_\Ecal^\ast\dolb_\Ecal$ on $\Ecal^k(M)$ for $k\leq n-1$,
        \item $i\Lcal_T = \Delta_{\dolb_\Ecal} - \Delta_{\partial_\Ecal}$ on $\Ecal^k(M)$ for $k\leq n-1$,
        \item $\sqrt{\Delta_\Ecal}:=d_\Ecal d_\Ecal^\ast + d_\Ecal^\ast d_\Ecal = \Delta_{\dolb_\Ecal} + \Delta_{\partial_\Ecal}$ on $\Ecal^k(M)$ for $k\leq n-1$,
        \item $\Delta_\Ecal$ commutes with $J$ on $\Ecal^k(M)$ for $k\leq n$,
        \item $\Lcal_T^\ast = -\Lcal_T$.
    \end{enumerate}
\end{lemma}
\noindent In the next lemma, we shall see why the condition $\Lcal_T(J)=0$ is necessary for $\Lcal_T$ to be compatible with the bigradation of $\frA^{\bullet,\bullet}(\Dcal)$.
\begin{lemma}
\label{contactlabel18}
    If the almost complex structure $J$ is integrable on $\Dcal$ and $\Lcal_T (J)=0$, then we have
    \begin{enumerate}
        \item \label{contactlabel18.1}$\Lcal_T\colon \frA^{p,q} (\Dcal)\rightarrow \frA^{p,q} (\Dcal)$,
        \item $\dolb_\Dcal^2=0=\partial_\Dcal^2$ and $\partial_\Dcal\dolb_\Dcal + \dolb_\Dcal\partial_\Dcal=-L\Lcal_T=-\Lcal_T L$,
        \item $\dolb_\Dcal\circ \Lcal_T = \Lcal_T\circ \dolb_\Dcal$ and $\dolb_\Dcal^\ast\circ \Lcal_T = \Lcal_T\circ \dolb_\Dcal^\ast$.
    \end{enumerate}
\end{lemma}
\begin{remark}
    From the previous result, we can see that the condition $\Lcal_T(J)=0$ is indeed essential for further calculations.
\end{remark}
\noindent According to \cite[Equations~(3.1),~(3.3)]{KitaokaHarmonicSasakian}, we also have
\begin{align}
\partial_\Dcal^\ast &= i[\Lambda,\dolb_\Dcal],
&\qquad
\dolb_\Dcal^\ast &= -i[\Lambda,\partial_\Dcal], \notag\\
\partial_\Dcal &= i[L,\dolb_\Dcal^\ast],
&\qquad
\dolb_\Dcal &= -i[L,\partial_\Dcal^\ast],\label{contactlabel38}
\end{align}
and $[\partial_\Dcal,\dolb_\Dcal^\ast]=0=[\dolb_\Dcal,\partial_\Dcal^\ast]$.\\

If $J$ is integrable and $\Lcal_T(J)=0$, then K\"ahler-type identities hold on the contact manifold $M$ for the operators arising, in particular, from the rescaled Rumin complex. In \cite[Chapter~II]{WeilKaehler} one finds the K\"ahler identities for a K\"ahler manifold $N$. The equations (\ref{contactlabel38}) correspond exactly to the equations in \cite[Equation~(VII)] {WeilKaehler}, where instead of $\partial_\Dcal,\dolb_\Dcal$ in the case of the K\"ahler manifold $N$, one of course uses $\partial,\dolb$ with $d=\partial+\dolb$ for the de Rham operator $d$ and the Dolbeault operator $\dolb$ on $N$. Furthermore, the operator $\Delta_\Ecal$ commutes with $J$, which is also the case for the Hodge Laplacian of the K\"ahler manifold, see \cite[Corollary~4.11]{WellsComplexManifolds}. However, there is a significant difference when it comes to the Laplacians. Let
\begin{align*}
    \Delta_{\partial} := \partial\partial^\ast + \partial^\ast\partial, && \Delta_{\dolb} := \dolb\dolb^\ast + \dolb^\ast\dolb
\end{align*}
on $\frA^{\bullet,\bullet}(N)$ and denote the Hodge Laplacian by $\Delta=dd^\ast+d^\ast d$ of $N$. The classical K\"ahler identities \cite[Theorem~2]{WeilKaehler} are stated as follows
\begin{align*}
    2\Delta = \Delta_{\partial} = \Delta_{\dolb} = \Delta_{\partial}+\Delta_{\dolb}.
\end{align*}
A comparison with the equations
\begin{enumerate}
    \item $i\Lcal_T = \Delta_{\dolb_\Ecal} - \Delta_{\partial_\Ecal}$ on $\Ecal^k(M)$ for $k\leq n-1$,
    \item $\sqrt{\Delta_\Ecal}:=d_\Ecal d_\Ecal^\ast + d_\Ecal^\ast d_\Ecal = \Delta_{\dolb_\Ecal} + \Delta_{\partial_\Ecal}$ on $\Ecal^k(M)$ for $k\leq n-1$,
\end{enumerate}
from Lemma~\ref{contactlabel8} shows that, in particular, $\Delta_{\dolb_\Ecal} \neq \Delta_{\partial_\Ecal}$ generally holds here, which differs from the K\"ahler identities where $\Delta_{\partial} = \Delta_{\dolb}$. However, for $\alpha\in\Ecal^k(M)$ with $\Lcal_T(\alpha)=0$, we do have $\Delta_{\dolb_\Ecal}\alpha = \Delta_{\partial_\Ecal}\alpha$. The additional direction along the Reeb vector field therefore makes the difference in the case of the contact manifold.\\

\noindent However, if $M$ is a compact Sasakian manifold, then Kitaoka showed in \cite[Theorem~1.1]{KitaokaHarmonicSasakian} that
\begin{align*}
    \kernel (\Delta\colon \frA^k(M)\rightarrow \frA^k (M)) = \kernel (\Delta_\Ecal\colon \Ecal^k(M)\rightarrow \Ecal^k(M)),
\end{align*}
where $\Delta$ is the Hodge Laplacian of $M$.\\
For compact homogeneous contact manifolds, we shall cite a result by Boothby \cite{contactmanifoldsNURboothby} in Section~\ref{contactsect8.1}, which states that for every compact homogeneous contact manifold $G/H$ we obtain a K\"ahler manifold $G/K$ with $H\subset K$ such that $G/H\rightarrow G/K$ is a $S^1$-principal bundle.
\section{Analytic Torsion of the Rumin Complex}
\label{contactsect7}
\subsection{Definition of the Analytic Contact Torsion}
\label{contactsect7.1}
Unlike the Hodge Laplacian the Rumin Laplacian is not elliptic, but it is maximally hypoelliptic on every compact contact manifold, according to \cite[p.~290]{ruminformesdiffer}. Therefore, the Rumin Laplacian is a self-adjoint operator, which possesses a discrete, real, non-negative spectrum and a smooth heat kernel $k_t(x,y)$. The heat kernel has an asymptotic on its diagonal $k_t(x,x)$ as $t\searrow 0$ according to \cite[Theorem~3.1]{rumin2012analytic}, that allows one to define a well\nobreakdash-defined zeta function. This asymptotic behaviour is crucial, as it guarantees that the integral defining the Mellin transform converges for $\operatorname{Re}(s)\gg 1$, and it ensures that the resulting zeta function admits a meromorphic continuation to $\C$, which is holomorphic in $s=0$.

We shall now recall the definition of the zeta function and other relevant results from \cite{rumin2012analytic} and \cite{rumin2026}.

\noindent The zeta function is defined for $\operatorname{Re} (s)\gg 1$ by
\begin{align*}
    \zeta(\Delta_\Ecal\vert \Ecal^k(M)) := \dim (H_\Ecal^k(M)) + \sum\limits_{\lambda\in\operatorname{spec}^\ast (\Delta_\Ecal\vert \Ecal^k(M))} \frac{1}{\lambda^s},
\end{align*}
where $\operatorname{spec}^\ast (\Delta_\Ecal\vert \Ecal^k(M))$ denotes the non-zero spectrum of $\Delta_\Ecal$ on $\Ecal^k(M)$.\\
For $\alpha\in \Ecal^k(M)$ and $x\in M$, let
\begin{align*}
    (e^{-t\Delta_\Ecal}\alpha)_{\vert x} := \int\limits_M k_t(x,y)\alpha_{\vert y}{\operatorname{dvol}_g}_{\vert y}
\end{align*}
denote the heat operator, which is trace class. The zeta function $\zeta(\Delta_\Ecal\vert \Ecal^k(M))$ can be meromorphically continued to $\C$ with the help of the Mellin transform. It is well-defined for $\operatorname{Re} (s)\gg 1$ and meromorphic with (at worst) simple poles at $s\in\lbrace \frac{n+1-j}{2}\mid j\in\N\rbrace\setminus (-\N)$. Therefore, we have
\begin{align*}
    \zeta(\Delta_\Ecal\vert \Ecal^k(M))(s) := \dim (H_\Ecal^k(M)) + \frac{1}{\Gamma (s)}\int\limits_0^{+\infty} \Tr^\ast(e^{-t\Delta_\Ecal}\vert \Ecal^k(M)) t^{s-1}dt,
\end{align*}
where $\Tr^\ast$ denotes the trace over the non-zero spectrum of $\Delta_\Ecal$ and $\Gamma (s)=\int\limits_0^{+\infty} t^{s-1}e^{-t}dt$ is the Gamma function.
\begin{remark} We have
    \begin{enumerate}
        \item $\lambda\in \operatorname{spec} (\Delta_\Ecal\vert \Ecal^k(M)) \Leftrightarrow e^{-t\lambda}\in \operatorname{spec} (e^{-t\Delta_\Ecal}\vert \Ecal^k(M))$,
        \item $\Tr^\ast(e^{-t\Delta_\Ecal}\vert \Ecal^k(M)) = \Tr(e^{-t\Delta_\Ecal}\vert \Ecal^k(M)) - \dim(H^k_\Ecal(M))$.
    \end{enumerate}
\end{remark}
The \textbf{contact zeta function} of the Rumin complex if then defined as
\begin{align*}
    Z(s) := \frac{1}{2} \sum\limits_{k=0}^{2n+1} (-1)^k \omega(k) \zeta(\Delta_\Ecal\vert \Ecal^k(M))(s),
\end{align*}
where 
\begin{align*}
    \omega(k):=\begin{cases}
    k &\text{ for } k\leq n,\\
    k+1 &\text{ for } k\geq n+1.
\end{cases}
\end{align*}
For $k\leq n$, we have
\begin{align*}
    &\frac{1}{2}((-1)^{k+1}\omega(k) + (-1)^{2n+1-k+1}\omega (2n+1-k))\\
    =&\frac{1}{2}(-1)^k(-k+2n+1-k+1) = (-1)^k(n+1-k).
\end{align*}
Given the properties of the Hodge $\ast$ operator from Lemma~\ref{contactlabel7}, the zeta function can also be expressed as follows
\begin{align*}
    Z(s)=\sum\limits_{k=0}^n (-1)^k (n+1-k) \zeta(\Delta_\Ecal\vert \Ecal^k(M))(s).
\end{align*}
The \textbf{analytic (contact) torsion} of the Rumin complex is defined by
\begin{align*}
    T_\Ecal (M) := {-\frac{1}{2}Z'(0)}.
\end{align*}
\begin{remark}
    As opposed to \cite{rumin2012analytic} and the further literature on this topic, we define torsion without the exponential function. Rumin defined torsion as $T_\Ecal (M) := \exp\left({-\frac{1}{2}Z'(0)}\right)$.
\end{remark}
For further details and the foundational framework regarding the Mellin transform, the construction of heat kernels, and their asymptotic expansions for classical differential operators such as the Hodge Laplacian, the reader is referred to \cite{HeatKernelsBerlineGetzlerVergne} and \cite{GilkeyHeatEquationIndexTheorem}.
\subsection{Analytic Contact Torsion on CR Seifert Manifolds}
\label{contactsect7.2}
In this section, we will summarise some key results from \cite{rumin2026}, as they will significantly simplify our calculations later on. The main point is that, under certain conditions, some terms in the formula for the zeta function cancel out, so that we only need to consider spaces on which the Rumin Laplacian corresponds exactly to $-\Lcal_T^2$. The equations from Lemma~\ref{contactlabel8} played a decisive role in \cite{rumin2026}, which once again highlights why it makes sense to work with the rescaled Rumin complex. A similar approach, which is also based on this lemma, can be found in \cite{KitaokaHarmonicSasakian}, which we will briefly discuss at the end of this section.
\begin{definition}[{\cite[Definition~3.1]{rumin2026}}]
    A compact contact manifold $M$ with an integrable complex structure on $\Dcal$ with $\Lcal_T(J)=0$ is called \textbf{CR Seifert manifold} if it admits a circle action generated by the Reeb vector field $T$.
\end{definition}
Now, let $M$ be a CR Seifert manifold with globally defined contact form. For $t>0$, the \textbf{heat torsion function} is defined by
\begin{align*}
    \vartheta (t) := \sum\limits_{k=0}^n (-1)^k (n+1-k)\Tr(e^{-t\Delta_\Ecal}\vert \Ecal^k(M)).
\end{align*}
We extend $\Delta_\Ecal$ from $\bigoplus\limits_{k=0}^n \Ecal^k(M)=\bigoplus\limits_{k=0}^n \frA^k_{\prim} (\Dcal)$ to the whole of $\frA^\bullet (\Dcal)$ using the Lefschetz decomposition (\ref{contactlabel9}), by requiring that $\Delta_\Ecal L = L\Delta_\Ecal$.\\
According to \cite{rumin2026}, we know that $\Delta_\Ecal$ then also commutes with $d_\Dcal$ and $d_\Dcal^\ast$.\\

A closer look at the Lefschetz decomposition shows, by Lemma~\ref{contactlabel10}, that $\frA^\bullet (\Dcal)$ contains $n+1-k$ copies of $\frA^k_{\prim} (\Dcal)$. For $j\in\N_0$, we have $L^j(\frA^\mathrm{odd}_{\prim} (\Dcal))\subset \frA^\mathrm{odd} (\Dcal)$ and $L^j(\frA^\mathrm{even}_{\prim} (\Dcal))\subset \frA^\mathrm{even} (\Dcal)$. Thus, the heat torsion function $\vartheta (t)$ can be rewritten as
\begin{align*}
    \vartheta (t) = & \sum\limits_{k=0}^{2n} (-1)^k \Tr(e^{-t\Delta_\Ecal}\vert \frA^k(\Dcal))\\
    = & \Tr (e^{-t\Delta_\Ecal}\vert \frA^\mathrm{even}(\Dcal)) - \Tr (e^{-t\Delta_\Ecal}\vert \frA^\mathrm{odd}(\Dcal)).
\end{align*}
Now consider $D_\Dcal:=d_\Dcal+d_\Dcal^\ast$. We have
\begin{align*}
    D_\Dcal\colon \frA^\mathrm{even}(\Dcal)\rightarrow \frA^\mathrm{odd}(\Dcal),\\
    D_\Dcal\colon \frA^\mathrm{odd}(\Dcal)\rightarrow \frA^\mathrm{even}(\Dcal),
\end{align*}
since $d_\Dcal$ increases the degree by $1$ and $d_\Dcal^\ast$ decreases the degree by $1$. Therefore,
\begin{align*}
    \vartheta (t) = \Tr (e^{-t\Delta_\Ecal}\vert \kernel (D_\Dcal)\cap\frA^\mathrm{even}(\Dcal)) - \Tr (e^{-t\Delta_\Ecal}\vert \kernel (D_\Dcal)\cap\frA^\mathrm{odd}(\Dcal)),
\end{align*}
since $D_\Dcal$ commutes with $\Delta_\Ecal$.
According to \cite[Proposition~3.4]{rumin2026}, we can express the operator $D_\Dcal$ with the help of $U:=e^{i\pi (L+\Lambda)/4}\colon \frA^\bullet_\C(\Dcal)\rightarrow \frA^\bullet_\C (\Dcal)$ as
\begin{align*}
    U^{-1}D_\Dcal U = \sqrt{2} (\dolb_\Dcal + \dolb_\Dcal^\ast)
\end{align*}
on $\frA^\bullet_\C(\Dcal)$. Since $\dolb_\Dcal^2=0$ by Lemma~\ref{contactlabel18}, it follows that
\begin{align*}
    \kernel({D_\Dcal}_{\vert \frA^\bullet_\C(\Dcal)}) = U(\kernel(\dolb_\Dcal+\dolb_\Dcal^\ast)) = U (\kernel (\Delta_{\dolb_\Dcal}))
\end{align*}
for $\Delta_{\dolb_\Dcal} := \dolb_\Dcal\dolb_\Dcal^\ast + \dolb_\Dcal^\ast\dolb_\Dcal$.
\begin{lemma}[{\cite[Lemma~3.3]{rumin2026}}]
    We have $\Delta_\Ecal = -\Lcal_T^2$ on $\kernel (D_\Dcal)$, and since $U$ commutes with $\Lcal_T$, we even have
    \begin{align*}
        \vartheta(t) =& \Tr (e^{t\Lcal_T^2}\vert \kernel (\Delta_{\dolb_\Dcal})\cap\frA^\mathrm{even}_\C(\Dcal)) - \Tr (e^{t\Lcal_T^2}\vert \kernel (\Delta_{\dolb_\Dcal})\cap\frA^\mathrm{odd}_\C(\Dcal))\\
        =& \sum\limits_{k=0}^{2n} (-1)^k \Tr(e^{t\Lcal_T^2}\vert \kernel (\Delta_{\dolb_\Dcal})\cap\frA^k_\C(\Dcal)).
    \end{align*}
    Both traces converge since they are parts of the traces of $e^{-t\Delta_\Ecal}$, which converge.
\end{lemma}
\noindent Taking the limit $\lim\limits_{k\rightarrow +\infty} \vartheta(t)$ therefore also implies that
\begin{align}
    A :=& \sum\limits_{k=0}^{2n} (-1)^k(n+1-k)\dim_\R(H_\Ecal^k(M)) \notag\\
    =& \sum\limits_{k=0}^{2n} (-1)^k \dim_\C (\kernel(-\Lcal_T^2)\cap \kernel (\Delta_{\dolb_\Dcal})\cap\frA^k_\C(\Dcal)), \label{contactlabel12}
\end{align}
which is precisely the constant term in the zeta function $Z(s)$. In particular, this means that the dimensions on the right-hand side of the equation are all finite. In the following sections, we will be able to interpret these dimensions with the help of certain Dolbeault cohomologies.\\

By Mellin transform, we obtain
\begin{align}
    Z(s) =& \sum\limits_{k=0}^n (-1)^k(n+1-k)\zeta(\Delta_\Ecal\vert \Ecal^k(M))(s) \notag\\
    =& A + \frac{1}{\Gamma(s)}\int\limits_0^{+\infty}\sum\limits_{k=0}^n (-1)^k(n+1-k) \Tr^\ast(e^{-t\Delta_\Ecal}\vert \Ecal^k(M))t^{s-1}dt \notag\\
    =& A + \frac{1}{\Gamma(s)}\int\limits_0^{+\infty}\sum\limits_{k=0}^{2n} (-1)^k \Tr^\ast(e^{t\Lcal_T^2}\vert \kernel (\Delta_{\dolb_\Dcal})\cap\frA^k_\C(\Dcal)) t^{s-1} dt \notag\\
    =& \sum\limits_{k=0}^{2n} (-1)^k \zeta (-\Lcal_T^2\vert \kernel (\Delta_{\dolb_\Dcal})\cap\frA^k_\C(\Dcal))(s), \label{contactlabel35}
\end{align}
where 
\begin{align*}
    \zeta (-\Lcal_T^2\vert \kernel (\Delta_{\dolb_\Dcal})\cap\frA^k_\C(\Dcal))(s) := &\dim_\C (\kernel(-\Lcal_T^2)\cap \kernel (\Delta_{\dolb_\Dcal})\cap\frA^k_\C(\Dcal))\\
    &+ \sum\limits_{\lambda\in\operatorname{spec}^\ast (-\Lcal_T^2\vert \kernel (\Delta_{\dolb_\Dcal})\cap\frA^k_\C(\Dcal))} \frac{1}{\lambda^s}.
\end{align*}
In Section~\ref{contactsect8.2}, we can precisely determine the eigenvalues of $-\Lcal_T^2$ and interpret the corresponding eigenspaces as Dolbeault cohomologies of a complex manifold associated with the contact manifold.\\

In \cite{KitaokaHarmonicSasakian}, Kitaoka found a similar result to that of Rumin in \cite{rumin2026} for Sasakian manifolds, i.e.\ contact manifolds such that $J$ is integrable and $\Lcal_T(J) = 0$, by using the results we have summarised in Lemma~\ref{contactlabel8}. It is shown that certain terms in the zeta function cancel out. To achieve this, one considers the operators $\Delta_{\partial_\Ecal}$ and $\Delta_{\dolb_\Ecal}$ to obtain the following decomposition for $0\leq k\leq n-1$:
\begin{align*}
    \Ecal^k(M) = &(\kernel (\Delta_{\partial_\Ecal}) \cap \kernel (\Delta_{\dolb_\Ecal})) \oplus (\kernel (\Delta_{\partial_\Ecal}) \cap \image (\Delta_{\dolb_\Ecal}))\\
    &\oplus (\image (\Delta_{\partial_\Ecal}) \cap \kernel (\Delta_{\dolb_\Ecal})) \oplus (\image (\Delta_{\partial_\Ecal}) \cap \image (\Delta_{\dolb_\Ecal})).
\end{align*}
Since these two operators commute according to \cite[Proposition~4.2]{KitaokaHarmonicSasakian}, $\Ecal^k(M)$ can be further decomposed into the joint eigenspaces
\begin{align*}
    Q^k(\lambda_{1,0},\lambda_{0,1}) := \lbrace \alpha\in\Ecal^k(M)\mid \Delta_{\partial_\Ecal} \alpha = \lambda_{1,0}\alpha, \Delta_{\dolb_\Ecal} \alpha = \lambda_{0,1}\alpha\rbrace
\end{align*}
for $\lambda_{1,0},\lambda_{0,1}\geq 0$.\\
Together with the equations from Lemma~\ref{contactlabel8}
\begin{itemize}
    \item $i\Lcal_T = \Delta_{\dolb_\Ecal} - \Delta_{\partial_\Ecal}$ on $\Ecal^k(M)$ for $k\leq n-1$,
    \item $\sqrt{\Delta_\Ecal}:=d_\Ecal d_\Ecal^\ast + d_\Ecal^\ast d_\Ecal = \Delta_{\dolb_\Ecal} + \Delta_{\partial_\Ecal}$ on $\Ecal^k(M)$ for $k\leq n-1$,
\end{itemize}
this leads to a number of cancellations in the zeta function due to \cite[Proposition~4.3]{KitaokaHarmonicSasakian}
\begin{align*}
    & Q^k(\lambda_{1,0},\lambda_{0,1})\\
    =& Q^k(\lambda_{1,0},\lambda_{0,1})\cap \image(\partial_\Ecal^\ast)\cap\image(\dolb_\Ecal^\ast) \oplus Q^k(\lambda_{1,0},\lambda_{0,1})\cap \image(\partial_\Ecal)\cap\image(\dolb_\Ecal^\ast)\\
    & \oplus Q^k(\lambda_{1,0},\lambda_{0,1})\cap \image(\partial_\Ecal^\ast)\cap\image(\dolb_\Ecal) \oplus Q^k(\lambda_{1,0},\lambda_{0,1})\cap \image(\partial_\Ecal)\cap\image(\dolb_\Ecal)\\
    =& Q^k(\lambda_{1,0},\lambda_{0,1})_1 \oplus Q^k(\lambda_{1,0},\lambda_{0,1})_2 \oplus Q^k(\lambda_{1,0},\lambda_{0,1})_3 \oplus Q^k(\lambda_{1,0},\lambda_{0,1})_4\\
    \subset & \image (\Delta_{\partial_\Ecal}) \cap \image (\Delta_{\dolb_\Ecal})
\end{align*}
and \cite[Proposition~4.4]{KitaokaHarmonicSasakian}
\begin{equation*}
    \begin{tikzcd}
        Q^{k-1}(\lambda_{1,0},\lambda_{0,1})_1 \arrow[dr, "\cong" {near end},"\partial_\Ecal" swap]\arrow[drr, "\cong","\dolb_\Ecal" swap] & Q^{k-1}(\lambda_{1,0},\lambda_{0,1})_2 & Q^{k-1}(\lambda_{1,0},\lambda_{0,1})_3 & Q^{k-1}(\lambda_{1,0},\lambda_{0,1})_4 \\
        Q^k(\lambda_{1,0},\lambda_{0,1})_1 & Q^k(\lambda_{1,0},\lambda_{0,1})_2\arrow[drr, "\cong","\dolb_\Ecal" swap] & Q^k(\lambda_{1,0},\lambda_{0,1})_3\arrow[dr, "\partial_\Ecal" , "\cong" {swap, near start}] & Q^k(\lambda_{1,0},\lambda_{0,1})_4\\
        Q^{k+1}(\lambda_{1,0},\lambda_{0,1})_1 & Q^{k+1}(\lambda_{1,0},\lambda_{0,1})_2 & Q^{k+1}(\lambda_{1,0},\lambda_{0,1})_3 & Q^{k+1}(\lambda_{1,0},\lambda_{0,1})_4
    \end{tikzcd}
\end{equation*}
for $\lambda_{1,0},\lambda_{0,1} > 0$, because we also have
\begin{align*}
    (-1)^{k-1}(n+1-(k-1))+2\cdot (-1)^k(n+1-k)+(-1)^{k+1}(n+1-(k+1))=0
\end{align*}
and
\begin{align*}
    \Delta_\Ecal = (\Delta_{\partial_\Ecal}+\Delta_{\dolb_\Ecal})^2 = (\lambda_{1,0}+\lambda_{0,1})^2>0.
\end{align*}
For a better overview, not all isomorphisms are shown with arrows.
Ultimately, only spaces remain on which $\Delta_\Ecal = -\Lcal_T^2$. In degrees $k\in\lbrace 0,\ldots,n-1\rbrace$, these are precisely the spaces $\kernel (\Delta_{\partial_\Ecal}) \cap \image (\Delta_{\dolb_\Ecal})$ and $\image (\Delta_{\partial_\Ecal}) \cap \kernel (\Delta_{\dolb_\Ecal})$. In degree $n$, the space that remains is not so easy to interpret.\\
For symmetric contact manifolds with an action of $S^1$, as we shall consider later, one can precisely determine $-\Lcal_T^2$ on these remaining spaces by decomposing them with respect to the action of $S^1$, so that, at least in the degrees $k \in \lbrace 0,\ldots, n-1\rbrace$, the problem is reduced mainly to the harmonic forms $\kernel (\Delta_{\partial_\Ecal})$ and $\kernel (\Delta_{\dolb_\Ecal})$.\\
In \cite{GarfieldLeeRuminComplex} and \cite{CaseRuminComplex}, the harmonic forms $\kernel (\Delta_{\dolb_\Ecal})$ were studied in greater detail, and it was shown that these spaces are isomorphic to the Kohn-Rossi cohomology.
\section{Symmetric Contact Manifolds}
\label{contactsect8}
\subsection{Homogeneous Contact Manifolds}
\label{contactsect8.1}
In \cite{kitaokaspheres}, the analytic torsion of the Rumin complex was computed for the spheres $S^{2n+1}$, which are symmetric contact manifolds. We aim to determine the torsion for general compact (semisimple) symmetric contact manifolds with a globally defined contact form. To achieve this, we first review the classification of such spaces from \cite{semisimplecontactalekseevsky} and subsequently describe the relevant spaces and operators using representation theory.
\begin{definition}[\cite{contactmanifoldsboothby}]
\label{contactlabel19}
    Let $(M,\Dcal)$ be a contact manifold with a globally defined contact form $\theta\in\frA^1(M)$. It is called a \textbf{homogeneous contact manifold} if there exists a connected Lie group $G$ which acts transitively and effectively as a group of diffeomorphisms on $M$ that leave $\theta$ invariant, i.e.\ $g^\ast \theta = \theta$.\\
    For $x\in M$ and the isotropy group $H$ of $x$, i.e.\ the elements of $G$ that fix $x$, we have $M\cong G/H$.
\end{definition}
In the following Lemma, Boothby and Wang proved a relationship between possible contact forms of a homogeneous manifold $G/H$ and certain left-invariant $1$-forms on the Lie group $G$. Note that the left-invariant differential forms (respectively vector fields) are induced precisely by elements of the dual space $\frg^\ast$ of the Lie algebra $\frg$. In our notation, we do not distinguish between the elements of Lie algebras and the left-invariant vector fields or differential forms induced by them. Let $\frh$ be the Lie algebra of $H$.
\begin{lemma}[{\cite[Lemma~2]{contactmanifoldsboothby}}]
\label{contactlabel24}
    There is a one-to-one correspondence between the homogeneous contact forms on $G/H$ of dimension $2n+1$, i.e.\ contact forms that are invariant under $G$ as in Definition~\ref{contactlabel19}, and left-invariant forms $\theta'\in\frA^1(G)$ with
    \begin{enumerate}
        \item $\theta'$ is invariant under $\Ad^\ast_{\vert H}$, which is the dual of the adjoint representation on $\frg$,
        \item $\theta'(\frh) = 0$,
        \item $(d\theta')^{n+1}=0$ and $\theta'\wedge(d\theta')^n\neq 0$.
    \end{enumerate}
\end{lemma}
\begin{lemma}[{\cite[Lemma~4,5]{contactmanifoldsboothby}}]
\label{contactlabel20}
    Let $G/H$ be a homogeneous contact manifold with contact form $\theta$ and $\pi_H\colon G \rightarrow G/H$, $g \mapsto gH$. For $\theta' := \pi_H^\ast \theta$, let $K := \lbrace k \in G \mid \Ad_k^\ast \theta' = \theta'\rbrace$. We have $H\subset K$. The Lie algebra of $K$ is $\frk = \lbrace X\in\frg\mid d\theta'(X,\frg)=0\rbrace$ and $\dim(\frk)=\dim(\frh)+1$.\\
    There is an $X\in \frk$ such that $T_g\pi_H(X_g) = T_{gH}$, where $T$ is the Reeb vector field associated to $\theta$. We shall also denote $X$ by $T$.
\end{lemma}
In the addendum \cite{contactmanifoldsNURboothby} to \cite{contactmanifoldsboothby}, Boothby generalised some results from the previous paper, including the following theorem.
\begin{theorem}[{\cite[Theorem~2.6]{contactmanifoldsNURboothby}}]
\label{contactlabel22}
    Every compact homogeneous contact manifold $G/H$ is a circle bundle over a compact homogeneous Hodge(-K\"ahler) manifold $G/K$, where $K$ is the Lie group defined in Lemma~\ref{contactlabel20}. The fibres of this bundle are exactly given by the integral curves of the flow $\Phi^T_\cdot$ of the Reeb vector field $T$.\\
    If $\omega$ is the K\"ahler form of $G/K$, then $\pi\colon G/H\rightarrow G/K$ is a principal bundle with connection $\theta$ and $\pi^\ast\omega = d\theta$.
\end{theorem}
According to \cite[p.~350]{contactmanifoldsNURboothby}, extensive conditions can be imposed on the Lie group $G$ without losing generality. Indeed, we can assume, without loss of generality, that $G$ is a compact, semisimple, simply connected Lie group.\\
Unlike in Lemma~\ref{contactlabel24}, we will not distinguish in our notation between the objects on $G/K$ and $G/H$ and the objects of the Lie algebras, since the structures on $G/K$ and $G/H$ are $G$-invariant in both cases.
\noindent We wish to consider symmetric contact manifolds, which are also homogeneous spaces associated with a semisimple Lie group $G$. These spaces were classified in \cite{semisimplecontactalekseevsky} using the results of \cite{BieliavskySymplectic}, but we must first clarify what symmetry means for contact manifolds.
\begin{definition}
    Let $(M,\Dcal)$ be a contact manifold.
    \begin{enumerate}
        \item A \textbf{contactomorphism} is a diffeomorphism $\varphi\colon M\rightarrow M$, such that $T\varphi(\Dcal)=\Dcal$.
        \item A contact manifold is said to be \textbf{symmetric} if, for every $x\in M$, there exists an involutive contactomorphism $\varphi$ such that $\varphi (x) = x$ and $T_x\varphi _{\vert \Dcal}=-\Id_\Dcal$. Such contactomorphism are called \textbf{symmetries}.
    \end{enumerate}
\end{definition}
\noindent In general, such symmetric contact manifolds are not homogeneous. Therefore, the classification in \cite{semisimplecontactalekseevsky} is restricted to those spaces $G/H$ which are also homogeneous and for which there exists a symmetry $\varphi$ at $eH\in G/H$ that normalises $G$, i.e.\ it must hold that $\varphi\circ g\circ \varphi^{-1}\in G$ for $g\in G$.\\
Since we are specifically interested in compact symmetric contact manifolds with a globally defined contact form, we restrict ourselves precisely to the compact spaces from the classification given in Table~5 in \cite{semisimplecontactalekseevsky}, see Table~\ref{contactlabel21}.
\begin{table}[h!]
\centering
\caption{Compact Hermitian CR symmetric contact spaces $G/H$}
\label{contactlabel21}
\begin{tabular}{|l|l|}
\hline
$\frg$ & $\frk = \frh \oplus \R$ \\
\hline
$\mathfrak{su}(n)$ & $\mathfrak{su}(p)\oplus \mathfrak{su}(n-p)\oplus \mathfrak{u}(1)$ \\
$\mathfrak{so}(2n)$ & $\mathfrak{su}(n)\oplus \mathfrak{so}(2)$ \\
$\mathfrak{so}(n)$ & $\mathfrak{so}(n-2)\oplus\mathfrak{so}(2)$ \\
$\mathfrak{sp}(n)$ & $\mathfrak{su}(n)\oplus \mathfrak{so}(2)$ \\
$\mathfrak{e}(6)$ & $\mathfrak{so}(10)\oplus \mathfrak{so}(2)$ \\
$\mathfrak{e}(7)$ & $\mathfrak{e}(6)\oplus \mathfrak{so}(2)$ \\
\hline
\end{tabular}
\end{table}
We are working with compact semisimple Lie groups, so the Killing form $B$ of the Lie group $G$ is negative definite according to \cite[Corollary 7.4.12]{kkdiffgeoENG}, which is why $-B$ is a canonical $\Ad_{\vert G}$-invariant scalar product on $\frg$.\\

According to \cite[Theorem~1.4]{semisimplecontactalekseevsky}, every such compact symmetric contact manifold is based on a simply connected compact symplectic symmetric base space $G/K$, as in Theorem~\ref{contactlabel22}. This base space $G/K$ can be represented as the adjoint orbit $\Ad_G(\xi)$ of an element $\xi\in\frg$, and $K=Z_G(\xi)$, which is connected. We have a symmetric decomposition $\frg = \frk \oplus \frp$, i.e.\ it arises from the eigenspaces of the symmetry at the point $eK$.\\
The symplectic form of $G/K$ is defined by
\begin{align*}
    \omega (X,Y) = d(B(\xi,\cdot )) (X,Y) = -B(\xi,[X,Y]) \text{ for } X,Y\in\frp.
\end{align*}
If the Lie group $G$ has several simple factors, then $G/K$ decomposes into a product $G/K=G_1/K_1\times\ldots\times G_m/K_m$, where $G_l$ are simple Lie groups and $K_l=K_l'\cdot Z(K_l)$ for a semisimple Lie group $K_l'$ and centre $Z(K_l)\cong S^1$ if $G/K$ is compact, and $\xi=\xi_1+\ldots + \xi_m$ with $0\neq \xi_l\in Z(\frk_l)$. For each factor, we have a decomposition $\frg_l=\frk_l\oplus \frp_l$ and $G_l/K_l$ is a simply connected symplectic manifold.\\
For each of these simple factors, the centre of $K_l$ is of real dimension $1$, and up to a scaling factor, one can uniquely choose $\xi_l$ such that $\ad_{\xi_l}\colon \frp_l\rightarrow \frp_l$ is a complex structure. This yields an integrable complex structure on $G_l/K_l$, which is Hermitian together with the scalar product $-B_{G_l}$. For $G := G_1 \times \ldots \times G_l$ and $K := K_1 \times \ldots \times K_l$, $G/K$ is precisely the decomposition above as a product, and everything (including the Killing form of $G$ and the complex structure induced by $\ad_{\xi} $ for $\xi:=\xi_1+\ldots + \xi_m$) is compatible with this factor structure. According to \cite{HelgasonDGLGSS}, $G/K$ is a K\"ahler manifold with the K\"ahler form being induced by $\omega (X,Y):=-B(X,JY)$ for $X,Y\in \frp$. In particular, the metric on $G/K$ is induced by $-B$. The metric on $G/K$ and the orientation provided by the complex structure yield, just as in Section~\ref{contactsect6.1}, an $L^2$-scalar product on $\Gamma(G/K,T(G/K))$ and $\frA^{\bullet,\bullet} (G/K,F)$ for any Hermitian vector bundle $F\rightarrow G/K$.\\

The above argument yields a symmetric contact manifold precisely if the orthogonal complement $\frh:=\kernel(B(\xi,\cdot))\cap\frk$ of $\xi$ in $\frk$ yields a Lie subgroup $H$ of $K$.\\
We have the following Lie bracket relations:
\begin{align}
    [\frk,\frk]\subset \frk,&& [\frk,\frp]\subset \frp, && [\frp,\frp]\subset \frk, \nonumber\\
    [\xi,\frk] = 0,&& [\xi,\frp]\subset \frp, && [\frh,\frh]\subset \frh. \label{contactlabel23}
\end{align}
For any pair of distinct factors $\frg_j$ and $\frg_l$ with $j\neq l$, we have $[\frg_k,\frg_l]=0$.\\
With respect to the Killing form $B$, the following holds:
\begin{align*}
    \R\cdot\xi \perp \frp, && \frh \perp \frp, && \R\cdot\xi \perp \frh, && \frk \perp \frp.
\end{align*}
According to \cite[Proposition~7.1]{semisimplecontactalekseevsky}, $\xi$ now generates $S^1$ as a Lie subgroup of $K$ that commutes with $H$, and furthermore, $K = (H\times S^1)/\Gamma$, where $\Gamma=H\cap S^1=\lbrace \exp_G(l\cdot \frac{2 \pi}{N}\xi)\mid 1\leq l\leq N\rbrace$ is a finite Lie subgroup of $S^1$, i.e.\ a group of unit roots for some $N\in\mathbb{N}$. The isomorphism is given by $H\times S^1\rightarrow K$, $(h,z)\mapsto hz$ with kernel $\lbrace (\gamma,\gamma^{-1})\mid \gamma\in H\cap S^1\rbrace$. The fact that one must also keep in mind at this point that the intersection of the two Lie groups might be non-trivial was overlooked in \cite[Proposition~7.1]{semisimplecontactalekseevsky}.\\
The contact structure $\Dcal$ is obtained by a $G$-invariant extension of $\frp$, and the contact form is induced by $\theta:=B(\xi,\cdot)$ as in Lemma~\ref{contactlabel24}. Just as on $G/K$, ${\ad_\xi}_{\vert \frp}$ induces a $G$-invariant integrable complex structure $J$ on $\Dcal$.\\
The flow of the Reeb vector field $T$ is given by $\Phi_t^T(gH)=g\exp_G(tT)H$, which is well\nobreakdash-defined in particular because $\exp(tT)\in S^1$ commutes with all elements of $H$. The action of $S^1$ on $G/H$ induced by the flow of $T$ thus corresponds to the action of $S^1$ from the right on $G/H$.
\begin{remark}
\label{contactlabel32}
    On $\frp$, we have ${\Ad_{\exp_G(\frac{\pi}{2}\xi)}}_{\vert \frp}=\cos (\frac{\pi}{2}) \Id+\sin (\frac{\pi}{2}) \ad_\xi=\ad_\xi$.\\
    In general,
    \begin{align}
        \Ad_{\exp_G(t\xi)}=\begin{pmatrix}
            \Id_\frk & 0 \\ 0 & \cos (t)\Id_\frp + \sin (t) \ad_\xi
        \end{pmatrix}, \label{contactlabel28}
    \end{align}
    and it is immediately clear that $\exp_G(t\xi) = 1 \Leftrightarrow t \in 2\pi$, i.e.\ $S^1=\R\cdot \xi/2\pi\mathbb{Z}\xi$. Furthermore, $\Ad_{\exp_G(t\xi)}$ has eigenvalues $e^{it}$ and $e^{-it}$ on $\frp\otimes\C$.
\end{remark}
\noindent Since $[\xi,\frk]=0$, we have ${\ad_\xi}_{\vert \frk}\equiv 0$. As $\ad_\xi$ operates blockwise, we obtain $-B(\xi,\xi)=-\Tr(\ad_\xi\circ\ad_\xi)=-\Tr (-\Id_{\vert \frp})=\dim\frp =:2n$.\\
The Reeb vector field corresponding to the contact form $\theta=B(\xi,\cdot)$ is therefore induced by $T:=-\frac{1}{\dim \frp}\xi=-\frac{1}{2n}\xi$.
\begin{remark}
    We have $-B\neq d\theta(\cdot,J\cdot)+\theta\otimes\theta$, because although
\begin{align*}
    d\theta (X,JY)=-B(\xi,[X,\ad_\xi Y])=-B([\xi,X],\ad_\xi Y)= -B(\ad_\xi X,\ad_\xi Y)=-B(X,Y),
\end{align*}
for $X,Y\in \frp$, we find that
\begin{align*}
    (\theta\otimes \theta)(\xi,\xi) = (-B(\xi,\xi))^2 && \Leftrightarrow && \theta\otimes\theta = -B(\xi,\xi)(-B)=-(\dim\frp)\cdot B = -2n\cdot B
\end{align*}
on $\R\cdot \xi$.
That is why $-B$ and $d\theta(\cdot,J\cdot)+\theta\otimes\theta$ differ on the subspace $\R\cdot\xi$. Still, this scaling does not affect the orthogonality.
\end{remark} 
\subsection{Representation-Theoretic Description of Operators and Differential Forms}
\label{contactsect8.2}
The study of differential forms and differential operators on a homogeneous space $G/K$ benefits from translating these objects into the language of representation theory. For a symmetric space $G/K$ with its associated decomposition $\frg = \frp \oplus \frk$, the space $\frA^k(G/K)$ can be identified with the space $C^\infty(G,\Lambda^k\frp^\ast)^K$ of $K$ equivariant functions on $G$. This identification simplifies the situation, since it reduces the analytic problems on the homogeneous space to more manageable computations with the help of Lie group and representation theory.\\
The decisive advantage of this lies in its compatibility with the natural differential operators like the de Rham operator. As demonstrated in \cite{ikedataniguchiLaplace}, the Hodge Laplacian $\Delta := dd^\ast +d^\ast d$ corresponds to the action of the Casimir operator of $G$ with respect to the isomorphism above. Since the action of the Casimir operator on irreducible representations of $G$ operators as scalar multiplication, the spectrum of the Hodge Laplacian can be determined much more easily.\\
This approach is particularly helpful in the context of Hermitian symmetric space since the Kodaira Laplacian $\Delta_{\dolb}:=\dolb\dolb^\ast + \dolb^\ast\dolb$ is related to the Hodge Laplacian $\Delta$, i.e.\ $2\Delta = \Delta_{\dolb}$, see \cite[Section~0.7]{griffiths1994principles}.\\
As shown in \cite{kkhermsymm} and \cite{kkReidemeisterTorsion}, the relation to the Casimir operator of $G$ allows for the purely representation-theoretic computation of the holomorphic torsion or the analytic torsion. The computation reduces to the algebraic determination of the Casimir eigenvalues and their multiplicities within the relevant representations.\\

In the following section, we will repeat the identification of the differential forms from \cite{ikedataniguchiLaplace} and use it to translate the spaces and operators relevant to our situation in terms of representation theory.

Let $G/H$ be a compact symmetric contact manifold with a corresponding base manifold $G/K$, as in the previous section.
Since we work with homogeneous spaces for which the Lie brackets satisfy the relations~(\ref{contactlabel23}), we can now proceed in exactly the same way as in \cite{ikedataniguchiLaplace} by using the decomposition of the Lie algebra $\frg = \frp \oplus (\R\cdot \xi)\oplus \frh$. Due to the $G$-invariance of the contact structure, complex structure and metric, we can describe the relevant spaces and operators much more simply using representations.\\
For a finite-dimensional complex/real representation $\rho\colon H\rightarrow \Aut(V)$ of $H$ (respectively $K$), we define
\begin{align*}
    C^\infty(G,V)^H:=\lbrace f\in C^\infty(G,V)\mid f(gh)=\rho(h^{-1})f(g)\text{ for all }h\in H\rbrace .
\end{align*}
The following map is an isomorphism
\begin{align*}
    \psi\colon \frA^k(G/H)\rightarrow C^\infty (G,\Lambda^k (\frg/\frh)^\ast)^H\cong C^\infty (G,\Lambda^k (\frp^\ast\oplus \R\cdot\theta))^H,
\end{align*}
which is defined by
\begin{align*}
    \psi (\alpha)_{\vert g}(Y_1,\ldots,Y_k):=&(\pi_H^\ast \alpha)_{\vert g}(T_eL_g(Y_1),\ldots,T_eL_g(Y_k))\\
    =&((\pi_H\circ L_g)^\ast\alpha)_{\vert e}(Y_1,\ldots,Y_k)
\end{align*}
for $\alpha\in\frA^k(G/H)$, $Y_1,\ldots,Y_k\in \frg$ and the projection $\pi_H\colon G\rightarrow G/H$. If one also considers an associated vector bundle $F := G \times_H V$ associated to a representation $V$ of $H$, then this isomorphism can be extended to an isomorphism $\frA^k(G/H, F) \rightarrow C^\infty(G, \Lambda^k(\frp^\ast \oplus \R \cdot \theta) \otimes V)^H$.
\begin{remark}
    Now $G$ acts on both spaces from the left, and the isomorphism above is equivariant with respect to these respective actions.
    \begin{itemize}
        \item On $\frA^k (G/H)$, $G$ acts by $g\cdot\alpha:=(L_{g^{-1}})^\ast \alpha$.
        \item On $C^\infty (G,\Lambda^k (\frg/\frh)^\ast)^H$, $G$ acts by $(g\cdot \alpha)_{\vert x}:=\alpha_{\vert g^{-1}x}$.
    \end{itemize}
    The space $C^\infty (G,V)^H$ is called the \textbf{induced representation of $G$} for a representation $V$ of $H$, see also \cite[Section~III.6]{broetomdieck}.
\end{remark}
\noindent This identification is very helpful when considering the matter further, as it allows us to describe maps and spaces using representation theory.\\
The contact form is $G$-invariant, thus
\begin{align*}
    \psi\colon \frA^k(\Dcal) &\rightarrow C^\infty (G,\Lambda^k \frp^\ast)^H,\\
    \psi\colon \theta\wedge\frA^k(\Dcal) &\rightarrow C^\infty (G,\R\cdot\theta\wedge\Lambda^k \frp^\ast)^H.
\end{align*}
$J=\ad_\xi$ is a complex structure on $\frp$, which, via $G$-invariant extensions, yields the complex structures on $\Dcal$ and $G/K$. We define
\begin{align*}
    \frp^{1,0}:=\kernel ((J-i\cdot\Id)_{\vert \frp\otimes_\R\C}),\\
    \frp^{0,1}:=\kernel ((J+i\cdot\Id)_{\vert \frp\otimes_\R\C}).
\end{align*}
Since the complex structure is also $G$-invariant, we can extend the isomorphism above to the complexified case as well:
\begin{align}
    \psi\colon \frA^{p,q}(\Dcal)\rightarrow C^\infty (G,\Lambda^{p,q} \frp^\ast)^H.\label{contactlabel4}
\end{align}
For $G/K$, we have, with the projection $\pi_K\colon G\rightarrow G/K$, a completely analogous isomorphism
\begin{align*}
    \psi\colon \frA^k(G/K,F)\rightarrow C^\infty (G,\Lambda^k (\frg/\frk)^\ast\otimes V)^H\cong C^\infty (G,\Lambda^k \frp^\ast\otimes V)^K,
\end{align*}
which we also denote by $\psi$, for a representation $V$ of $K$ and $F=G\times_K V$, due to $\pi_K^\ast F\cong G\times V$.
For a complex representation $V$, we have the isomorphism
\begin{align*}
    \psi\colon \frA^{p,q}(G/K,F)\rightarrow C^\infty (G,\Lambda^{p,q} \frp^\ast\otimes V)^K.
\end{align*}
\noindent Let $\operatorname{dvol}_G$ be the left-invariant volume form on $G$ induced by $-B$. This yields a (Hermitian) scalar product on the induced $G$-representations together with the respective scalar product $\metric{\cdot}{\cdot}$ of the representation $W$ of $H$ (respectively $K$), by defining 
\begin{align}
    \metric{\alpha}{\beta}':=\int\limits_G \metric{\alpha_{\vert g}}{\beta_{\vert g}} \operatorname{dvol}_G \label{contactlabel33}
\end{align}
for $\alpha,\beta\in C^\infty (G,W)^H$ (respectively $\alpha,\beta\in C^\infty (G,W)^K$).\\
If $V$ is a complex representation with $K$-invariant Hermitian scalar product, then this yields a $G$-invariant Hermitian metric on $F=G\times_K V$. For the corresponding $L^2$-metric, we have
\begin{align}
    \metric{\psi(\alpha)}{\psi(\beta)}'=c\cdot \metric{\alpha}{\beta}_{L^2} \label{contactlabel25}
\end{align}
for $\alpha,\beta\in\frA^{p,q} (G/K,F)$ and some constant $c>0$, see \cite[Equations~(2),~(3)]{ikedataniguchiLaplace}.
In this case, the scalar product on $\Lambda^{p,q}\frp^\ast$ is induced by $d\theta(\cdot,J\cdot)=-B$. Exactly the same result holds for $\frA^{p,q}(\Dcal)$ and the corresponding induced $G$-representation.

For 
\begin{align*}
    &W_1,W_2\subset \frA^{\bullet}_\C(G/H):=\Gamma(G/H,\Lambda^\bullet (T^\ast G/H)\otimes \C),\\
    (\text{respectively }&W_1,W_2\subset \frA^\bullet_\C (G/K,F):=\Gamma(G/H,\Lambda^\bullet ((T^\ast G/K)\otimes \C)\otimes F)),
\end{align*}
and $\varphi\colon W_1\rightarrow W_2$, we denote $\widetilde{\varphi} := \psi \circ \varphi \circ \psi^{-1}$.\\
Now, $\widetilde{\iota_T} := \psi\circ \iota_T\circ\psi^{-1} = \iota_T$, where the last $\iota_T$ is the linear map $\Lambda^k (\frp^\ast\oplus \R\cdot \xi)\rightarrow \Lambda^{k-1} \frp^\ast$. Thus, for $\alpha\in C^\infty(G,\Lambda^k\frp^\ast)^H$, we have $\widetilde{\Lcal_T}\alpha = \iota_T\circ \widetilde{d}\alpha$. In general, $G$-invariant maps on the differential forms of $G/H$ or $G/K$ translate to the corresponding linear maps on the representations that induced these $G$-invariant maps in the first place, and we do not distinguish between them in our notation.\\
According to \cite[Theorem~2.3.8]{kkdiffgeoENG}, for a manifold $M$ and $Y_0,\ldots,Y_k\in\Gamma(M,TM)$, $\alpha\in\frA^k(M)$, we have
\begin{align*}
    (d\alpha)(Y_0,\ldots,Y_k) =& \sum\limits_{j=0}^k (-1)^j \Lcal_{Y_j} (\alpha(Y_0,\ldots,\widehat{Y}_j,\ldots,Y_k))\\
    & + \sum\limits_{0\leq j< l\leq k} (-1)^{j+l} \alpha([Y_j,Y_l],Y_0,\ldots,\widehat{Y}_j,\ldots,\widehat{Y}_l,\ldots,Y_k),
\end{align*}
where $\widehat{Y}$ indicates that $Y$ is omitted. Following \cite{ikedataniguchiLaplace}, for $\alpha\in\frA^k(G/H)$ (respectively $\alpha\in\frA^k(G/K)$) and $Y_0,\ldots,Y_k\in\frg$ we have
\begin{align*}
    &\psi(d\alpha)_{\vert g} (Y_0,\ldots,Y_k) \\
    =& (\pi^\ast (d\alpha))_{\vert g} ({Y_0}_{\vert g},\ldots,{Y_k}_{\vert g})\\
    =& ( d(\pi^\ast\alpha))_{\vert g} ({Y_0}_{\vert g},\ldots,{Y_k}_{\vert g})\\
    =& \sum\limits_{j=0}^k (-1)^j \Lcal_{Y_j} ((\pi^\ast\alpha)(Y_0,\ldots,\widehat{Y}_j,\ldots,Y_k))_{\vert g}\\
    & + \sum\limits_{0\leq j< l\leq k} (-1)^{j+l} (\pi^\ast\alpha)([Y_j,Y_l],Y_0,\ldots,\widehat{Y}_j,\ldots,\widehat{Y}_l,\ldots,Y_k)_{\vert g},
\end{align*}
where $\Lcal_{Y_j}$ denotes the Lie derivative along the left-invariant vector field ${Y_j}_{\vert g}:=T_eL_g(Y_j)$ on $G$. Thus, for $\alpha\in C^\infty (G,\Lambda^k (\frp^\ast\oplus \R\cdot\theta))^H$ (respectively $\alpha\in C^\infty (G,\Lambda^k \frp^\ast)^K$) and $Y_0,\ldots,Y_k\in\frg$, this leads to
\begin{align}
    (\widetilde{d}\alpha)_{\vert g}(Y_0,\ldots,Y_k) =& \sum\limits_{j=0}^k (-1)^j \Lcal_{Y_j} (\alpha(Y_0,\ldots,\widehat{Y}_j,\ldots,Y_k))_{\vert g}\notag\\
    & + \sum\limits_{0\leq j< l\leq k} (-1)^{j+l} \alpha([Y_j,Y_l],Y_0,\ldots,\widehat{Y}_j,\ldots,\widehat{Y}_l,\ldots,Y_k)_{\vert g},\label{contactlabel26}
\end{align}
where $[Y_j,Y_l]$ is the Lie bracket of the Lie algebra $\frg$.\\
In our case, $[\frp,\frp]\subset \frk$ and $[\xi,Y]=J(Y)\in\frp$ for $Y\in\frp$, which is why the second sum with the Lie brackets either disappears or is greatly simplified. The vanishing of the Lie brackets was essential in \cite{ikedataniguchiLaplace} in order to be able to express the Hodge Laplacian operator as the Casimir operator.
\begin{remark}
    For a left-invariant vector field $X$ on $G$ and $\alpha\in C^\infty (G,V)$ with a vector space $V$, one has
        \begin{align}
            \Lcal_X(\alpha)_{\vert g}=T_g\alpha(T_eL_g(X))=T_g\alpha(T_0\exp_G(\cdot X)(\frac{\partial}{\partial t}))=\frac{\partial}{\partial t}_{\vert t=0} \alpha_{\vert g\exp_G(tX)}.\label{contactlabel30}
        \end{align}
\end{remark}
Working with induced $G$-representations offers computational advantages, which is why we will use them as our primary framework in what follows. However, for the sake of clarity, we will refrain from introducing new symbols for the corresponding operators with respect to this identification if the operator has already been induced by a map on the underlying representation.
If, in individual cases, the context is not sufficient to make a clear distinction, we will explicitly point this out.
\begin{lemma}
    We have $\Lcal_T\circ J = J\circ \Lcal_T$, which implies $\Lcal_T(J)=0$.
\end{lemma}
\begin{proof}
    Since $J(T)=0$ and $\Lcal_T$ preserves the decomposition~(\ref{contactlabel1}), this is clearly the case for $\theta\wedge\frA^k(\Dcal)$. Let $\alpha\in C^\infty(G,\Lambda^k\frp^\ast)^H\cong \frA^k(\Dcal)\subset \kernel(\iota_T)$. 
    Due to the Cartan homotopy formula $\Lcal_T=d\circ \iota_T + \iota_T\circ d$ \cite[Theorem~2.3.7]{kkdiffgeoENG} and (\ref{contactlabel26}), for $Y_1,\ldots,Y_k\in\frp$, we obtain
    \begin{align*}
        &(\widetilde{\Lcal_T}\circ \widetilde{J}(\alpha))(Y_1,\ldots,Y_k) = \widetilde{d}(J(\alpha))(T,Y_1,\ldots,Y_k)\\
        =& \Lcal_T(\beta(J(Y_1),\ldots,J(Y_k))) + \sum\limits_{1\leq j\leq k}(-1)^j\alpha (J([T,Y_j]),J(Y_1),\ldots,\widehat{J(Y_j)},\ldots,J(Y_k))\\
        =& \Lcal_T(\beta(J(Y_1),\ldots,J(Y_k))) + \sum\limits_{1\leq j\leq k}(-1)^j\alpha ([T,J(Y_j)],J(Y_1),\ldots,\widehat{J(Y_j)},\ldots,J(Y_k))\\
        =& (\widetilde{d}\alpha)(T,J(Y_1),\ldots,J(Y_k)) = (\widetilde{J}\circ \widetilde{\Lcal_T}(\alpha))((Y_1,\ldots,Y_k)),
    \end{align*}
    since all other terms in (\ref{contactlabel26}) vanish when $T$ is inserted in $\alpha$ and
    \begin{align*}
        J([T,Y_j]) = \ad_\xi([T,Y_j])=[\underbrace{\ad_\xi(T)}_{=0},Y_j]+[T,\ad_\xi(Y_j)] = [T,J(Y_j)].
    \end{align*}
    From $\Lcal_T\circ J = J\circ \Lcal_T$, it follows immediately that $\Lcal_T(J)=0$.
\end{proof}
\noindent In particular, this shows that the compact symmetric contact manifolds that we are considering are actually CR Seifert manifolds, which is why we can apply the results from Section~\ref{contactsect7.2} in the following.

To determine the eigenvalues of $\Lcal_T$ on $\frA^{\bullet,\bullet}(\Dcal)$, we recall that $T$ generates the action of $S^1$ on $G/H$, therefore we wish to decompose the induced $G$-representations $C^\infty(G,\Lambda^{p,q}\frp^\ast)^H$ with respect to the action of $S^1$.

According to \cite[Section~III.6]{broetomdieck}, we have a decomposition of $C^\infty(G,\Lambda^{p,q}\frp^\ast)^H$ into irreducible representations
\begin{align*}
    C^\infty(G,\Lambda^{p,q}\frp^\ast)^H = \widehat{\bigoplus_\lambda} V_\lambda \otimes \Hom_H(V_\lambda,\Lambda^{p,q}\frp^\ast),
\end{align*}
where $\rho_\lambda\colon G\rightarrow \Aut(V_\lambda)$ is an irreducible complex representation of $G$.\\
The subspace $V_\lambda \otimes \Hom_H(V_\lambda,\Lambda^{p,q}\frp^\ast)$ is embedded by $v\otimes \varphi\mapsto (g\mapsto \varphi(\rho_\lambda(g^{-1})v))$.\\
The irreducible representations of $S^1$ are given by the characters $\exp_G(t\xi)\mapsto e^{m\cdot it}$ for $m\in\Z$, according to \cite[Proposition~II.8.1]{broetomdieck}. By decomposing each $V_\lambda$ into irreducible representations of $S^1\subset G$, one obtains the decomposition
\begin{align}
    C^\infty(G,\Lambda^{p,q}\frp^\ast)^H = \widehat{\bigoplus\limits_{m\in\Z}} C^\infty(G,\Lambda^{p,q}\frp^\ast)^H_m, \label{contactlabel34}
\end{align}
where
\begin{align*}
    C^\infty(G,\Lambda^{p,q}\frp^\ast)^H_m := \left\{ \alpha\in C^\infty(G,\Lambda^{p,q}\frp^\ast)^H \mid \forall t\in\R,g\in G:\alpha_{\vert g\exp_G(t\xi)} = e^{-m\cdot it}\alpha_{\vert g} \right\}.
\end{align*}

We describe the action of $S^1$ on $\Lambda^{p,q}\frp^\ast$. Due to (\ref{contactlabel28}), we immediately obtain
\begin{align}
    \exp_G(t\xi)\cdot \alpha = e^{(q-p)\cdot it}\alpha \text{ for } \alpha \in \Lambda^{p,q}\frp^\ast. \label{contactlabel29}
\end{align}
However, given that $K=(H\times S^1)/\Gamma$ with $\Gamma := H\cap S^1 = \lbrace \exp_G(l\cdot \frac{2 \pi}{N}\xi)\mid 1\leq l\leq N\rbrace$, there is a restriction on the possible values of $m\in\Z$:
\begin{lemma}
    For $m\in \Z$ and $p,q\geq 0$ with $m+p-q\not\in N\cdot \Z$, we have
    $C^\infty(G,\Lambda^{p,q}\frp^\ast)^H_m = 0$.
\end{lemma}
\begin{proof}
    For $\alpha\in C^\infty(G,\Lambda^{p,q}\frp^\ast)^H_m$ and the generator $\gamma:=\exp_G(\frac{2 \pi}{N}\xi)$ of $\Gamma$, it follows from $H$-equivariance and the $S^1$-weight on this space that
    \begin{enumerate}
        \item $\alpha_{\vert g\gamma} = \gamma^{-1}\cdot \alpha_{\vert g} = e^{(p-q)\cdot i\frac{2\pi}{N}}\alpha_{\vert g}$,
        \item $\alpha_{\vert g\gamma} = e^{m\cdot i\frac{2\pi}{N}}\alpha_{\vert g}$.
    \end{enumerate}
    Therefore, due to 
    \begin{align*}
        e^{(p-q)\cdot i\frac{2\pi}{N}} = e^{m\cdot i\frac{2\pi}{N}} \Leftrightarrow 1 = e^{\frac{m+p-q}{N}\cdot 2\pi i},
    \end{align*}
    the possible values of $m\in \Z$ yielding non-trivial spaces are restricted to those with $m+p-q\in N\cdot \Z$.
\end{proof}

\noindent As we shall see from the following results, $\widetilde{\Lcal_T}$ acts as a scalar on the spaces $C^\infty(G,\Lambda^{p,q}\frp^\ast)^H_m$.
\begin{lemma}
    \label{contactlabel31}
    For $\alpha\in C^\infty(G,\Lambda^{p,q} \frp^\ast)^H$ and $Y_1,\ldots,Y_k\in\frp$, $k=p+q$, we have
    \begin{align*}
        (\widetilde{\Lcal_T}\alpha) (Y_1,\ldots,Y_k) = -\frac{1}{2n} \Big( \Lcal_\xi(\alpha(Y_1,\ldots,Y_k)) + (T_e\rho_{p,q}(\xi)\alpha)(Y_1,\ldots,Y_k) \Big),
    \end{align*}
    where $\rho_{p,q}$ is the dual adjoint representation on the $(p,q)$-forms.
\end{lemma}
\begin{proof}
    We have $\alpha\in C^\infty(G,\Lambda^{p,q} \frp^\ast)^H\cong \frA^{p,q}(\Dcal)\subset \kernel(\iota_T)$. Therefore, due to the Cartan homotopy formula $\Lcal_T=d\circ \iota_T + \iota_T\circ d$ \cite[Theorem~2.3.7]{kkdiffgeoENG}, we obtain
    \begin{align*}
        (\widetilde{\Lcal_T}\alpha) (Y_1,\ldots,Y_k) =& (\widetilde{d}\alpha) (T,Y_1,\ldots,Y_k)\\
        =& \Lcal_T(\alpha(Y_1,\ldots,Y_k))+\sum\limits_{1\leq j\leq k} (-1)^j \alpha([T,Y_j],Y_1,\ldots,\widehat{Y_j},\ldots, Y_k)\\
        =& \Lcal_T(\alpha(Y_1,\ldots,Y_k)) + \sum\limits_{1\leq j\leq k} -\alpha(Y_1,\ldots,[T,Y_j],\ldots, Y_k)\\
        =& -\frac{1}{2n}(\Lcal_\xi(\alpha(Y_1,\ldots,Y_k)) + \sum\limits_{1\leq j\leq k} -\alpha(Y_1,\ldots,\underbrace{[\xi,Y_j]}_{=J(Y_j)},\ldots, Y_j))\\
        =& -\frac{1}{2n} \Big( \Lcal_\xi(\alpha(Y_1,\ldots,Y_k)) + (T_e\rho_{p,q}(\xi)\alpha)(Y_1,\ldots,Y_k )\Big).
    \end{align*}
    The latter is obtained as follows: For a $1$-form $\beta\in\frg^\ast$ and $Y\in\frg$, we have
    \begin{align*}
        (\ad_\xi^\ast (\beta))(Y)= \frac{\partial}{\partial t}_{\vert t=0} \beta(\Ad_{\exp_G(-t\xi)}Y) = \beta(\frac{\partial}{\partial t}_{\vert t=0}\Ad_{\exp_G(-t\xi)}Y) = -\beta([\xi,Y]).
    \end{align*}
    By applying the product rule, and since the action is applied to each factor, we obtain the desired result.
\end{proof}
\begin{corollary}
    For $\alpha\in C^\infty(G,\Lambda^{p,q} \frp^\ast)^H_m$, we obtain
    \begin{align*}
        \widetilde{\Lcal_T}\alpha = \frac{m+p-q}{2n}\cdot i \alpha.
    \end{align*}
    Therefore,
    \begin{align}
        \widetilde{-\Lcal_T^2}\alpha = \left(\frac{m+p-q}{2n}\right)^2 \alpha. \label{contactlabel6}
    \end{align}
\end{corollary}
\begin{proof}
    Let $\alpha\in C^\infty(G,\Lambda^{p,q} \frp^\ast)^H_m$. By restricting to $C^\infty(G,\Lambda^{p,q} \frp^\ast)^H_m$, we can refine the result from Lemma~\ref{contactlabel31} even further, since this depends on the action of $S^1$ from the right, so using (\ref{contactlabel30}) and the definition of $C^\infty(G,\Lambda^{p,q} \frp^\ast)^H_m$ for $Y_1,\ldots,Y_k\in\frp$, we immediately obtain
    \begin{align*}
        \Lcal_\xi(\alpha(Y_1,\ldots,Y_k)) =& \frac{\partial}{\partial t}_{\vert t=0} \alpha_{\vert g\exp_G(t\xi)} (Y_1,\ldots,Y_k)\\
        =& \frac{\partial}{\partial t}_{\vert t=0} e^{-m\cdot it}\alpha_{\vert g} (Y_1,\ldots,Y_k)\\
        =& -m\cdot i \alpha_{\vert g} (Y_1,\ldots,Y_k).
    \end{align*}
    Regardless of the $S^1$-weight $m$, it follows from (\ref{contactlabel29}) that
    \begin{align*}
        (T_e\rho_{p,q}(\xi)\alpha) = (q-p)\cdot i\alpha.
    \end{align*}
    In summary, this gives us
    \begin{align*}
        \widetilde{\Lcal_T}\alpha = \frac{m+p-q}{2n}\cdot i \alpha.
    \end{align*}
\end{proof}
Thus, we have determined $-\Lcal_T^2$ on the entire space $\frA^{\bullet,\bullet}(\Dcal)$ and, since we also know that $\widetilde{-\Lcal_T^2}$ acts a scalar on $C^\infty(G,\Lambda^{p,q} \frp^\ast)^H_m$, we wish to examine the space $C^\infty(G,\Lambda^{p,q} \frp^\ast)^H_m$ in greater detail.\\
Since we are applying the results from Section~\ref{contactsect7.2}, we are in fact only interested in $\kernel (\Delta_{\dolb_\Dcal})\cap\frA^{p,q}(\Dcal)_m$ for $\frA^{p,q}(\Dcal)_m:=\psi^{-1}(C^\infty(G,\Lambda^{p,q} \frp^\ast)^H_m)$. However, it turns out that this space corresponds to a certain Dolbeault cohomology on $G/K$, as we shall see below.\\
First, we shall interpret the space $C^\infty(G,\Lambda^{p,q} \frp^\ast)^H_m$ with the help of the base manifold $G/K$, since the transition to $K$\nobreakdash-equivariance of the functions in this space is, in fact, almost already given, as we know how the action of $S^1$ from the right behaves and we have $K=(H\times S^1)/\Gamma$. For $m+p-q=0$, we know that
\begin{align*}
    C^\infty(G,\Lambda^{p,q}\frp^\ast)^H_{q-p} =  C^\infty(G,\Lambda^{p,q}\frp^\ast)^K.
\end{align*}
However, for $0\neq m+p-q\in N\cdot \Z$, the $S^1$-weight $m$ does not correspond to the actual $S^1$-weight $q-p$ of $\Lambda^{p,q}\frp^\ast$, which is why we must determine how to compensate this. Since $S^1$-weights are integers, this can easily be achieved by adding a suitable one-dimensional complex representation for compensation.\\

Let $m+p-q\in N\cdot \Z$. We consider the $S^1$-representation $L^{m+p-q}:=\C$ where $\exp_G(t\xi)\cdot v:= e^{(m+p-q)\cdot it}v$ for $\exp_G(t\xi)\in S^1$. For the generating element $\gamma=\exp_G(\frac{2 \pi}{N}\xi)$ of $\Gamma$, we then have $\gamma\cdot v=e^{\frac{2 \pi(m+p-q)}{N}\cdot i}v=v$, which is why this representation of $S^1$ extends to a representation of $K=(H\times S^1)/\Gamma$, which we also denote by $L^{m+p-q}$. We equip this representation with the standard scalar product of $\C$, which is clearly invariant under the action of $K$. We choose $v:=1\in L^{m+p-q}$, which has norm $1$. By 
\begin{align*}
    \Lcal^{m+p-q}:=G\times_K L^{m+p-q}
\end{align*}
we denote the associated holomorphic line bundle, see \cite{tirao1970homogenous}.
\begin{lemma}
    Let $m\in\Z$ such that $m+p-q\in N\cdot \Z$. Then
    \begin{align*}
        \mu\colon C^\infty(G,\Lambda^{p,q}\frp^\ast)^H_m &\rightarrow C^\infty(G,(\Lambda^{p,q}\frp^\ast)\otimes L^{m+p-q})^K\\
        \alpha &\mapsto \alpha\otimes v
    \end{align*}
    is an isomorphism. This isomorphism translates to an isomorphism of $\frA^{p,q}(\Dcal)_m$ and $\frA^{p,q}(G/K,\Lcal^{m+p-q})$.
\end{lemma}
\begin{proof}
    Since $H$ acts trivially on $L^{m+p-q}$, we only need to check the $S^1$-equivariance.\\
    For $\alpha \in C^\infty(G,\Lambda^{p,q}\frp^\ast)^H_m$, we have
    \begin{align*}
        \mu (\alpha)_{\vert g\exp_G(t\xi)}=&\alpha_{\vert g\exp_G(t\xi)}\otimes v = e^{-m\cdot it} \alpha_{\vert g}\otimes v\\
        =& e^{-(q-p)\cdot it}\cdot e^{-(m+p-q)\cdot it} \alpha_{\vert g}\otimes v\\
        =& \exp_G(-t\xi)\cdot (\alpha_{\vert g}\otimes v)\\
        =& (\exp_G(t\xi))^{-1}\cdot (\alpha_{\vert g}\otimes v),
    \end{align*}
    where $e^{-(q-p)\cdot it}$ results from the action on $\Lambda^{p,q}\frp^\ast$ and $e^{-(m+p-q)\cdot it}$ results from the action on $L^{m+p-q}$. The second term therefore provides the necessary correction. Overall, $\mu$ is well-defined.\\
    Every element of $C^\infty(G,(\Lambda^{p,q}\frp^\ast)\otimes L^{m+p-q})^K$ can be written as $\alpha\otimes v$ for some $\alpha\in C^\infty(G,\Lambda^{p,q}\frp^\ast)^H_m$. Thus, it directly follows that $\mu$ is an isomorphism.
\end{proof}
\noindent Conversely, for $j\in N\cdot \Z$, $C^\infty(G,\Lambda^{p,q}\frp^\ast)^H_{j-p+q}\cong C^\infty(G,(\Lambda^{p,q}\frp^\ast)\otimes L^j)^K$.\\
Just as in (\ref{contactlabel27}), let
\begin{align*}
    X_1,\ldots,X_n,X_{n+1}=J(X_1),\ldots,X_{2n}=J(X_n)
\end{align*}
be a basis for $\frp$, and the corresponding complex bases are
\begin{align*}
    (X_j^{1,0}:=X_j-iX_{n+j})_{1\leq j\leq n} &\text{ of } \frp^{1,0}\\
    \text{and } (X_j^{0,1}:=X_j+iX_{n+j})_{1\leq j\leq n} &\text{ of } \frp^{0,1}.
\end{align*}
For $\alpha\in C^\infty(G,\Lambda^k\frp^\ast)^K$ and $Y_0,\ldots,Y_k\in\frp$, we obtain
\begin{align*}
    (\widetilde{d}\alpha)_{\vert g}(Y_0,\ldots,Y_k) = \sum\limits_{j=0}^k (-1)^j \Lcal_{Y_j} (\alpha(Y_0,\ldots,\widehat{Y}_j,\ldots,Y_k))_{\vert g},
\end{align*}
from (\ref{contactlabel26}), since $[Y_j,Y_l]\in \frk$. 
We extend the Lie derivative to the complexified case by defining $\Lcal_{X+iY}:=\Lcal_X + i\Lcal_Y$ for $X,Y\in\frg$. This allows us to describe the \mbox{Dolbeault} operator using the bases of $\frp^{1,0}$ and $\frp^{0,1}$ above. For $\alpha\in C^\infty(G,\Lambda^{p,q} \frp^\ast)^K$ and $1\leq l_1,\ldots,l_p,k_0,\ldots,k_q\leq n$,
\begin{align*}
    &(\widetilde{\dolb}\alpha)(X^{1,0}_{l_1},\ldots,X^{1,0}_{l_p},X^{0,1}_{k_0},\ldots,X^{0,1}_{k_q})\\
    =& \sum\limits_{j=0}^q (-1)^{p+j} \Lcal_{X^{0,1}_{k_j}}(\alpha(X^{1,0}_{l_1},\ldots,X^{1,0}_{l_p},X^{0,1}_{k_0},\ldots,\widehat{X^{0,1}_{k_j}},\ldots,X^{0,1}_{k_p})).
\end{align*}
In exactly the same way, we can also describe $\widetilde{\dolb_\Dcal}$ on $C^\infty(G,\Lambda^{p,q}\frp^\ast)^H$, because for $\alpha\in C^\infty(G,\Lambda^{p,q} \frp^\ast)^H$ and $1\leq l_1,\ldots,l_p,k_0,\ldots,k_q\leq n$,we have
\begin{align*}
    &(\widetilde{\dolb_\Dcal}\alpha)(X^{1,0}_{l_1},\ldots,X^{1,0}_{l_p},X^{0,1}_{k_0},\ldots,X^{0,1}_{k_q})\\
    =& \sum\limits_{j=0}^q (-1)^{p+j} \Lcal_{X^{0,1}_{k_j}}(\alpha(X^{1,0}_{l_1},\ldots,X^{1,0}_{l_p},X^{0,1}_{k_0},\ldots,\widehat{X^{0,1}_{k_j}},\ldots,X^{0,1}_{k_p})).
\end{align*}
We now would like to investigate how $\widetilde{\dolb_\Dcal}$ behaves on $C^\infty(G,\Lambda^{p,q} \frp^\ast)^H_m$ for $m+p-q\in N\cdot \Z$. For this, let $\alpha\in C^\infty(G,\Lambda^{p,q} \frp^\ast)^H_m$ and $z:=\exp_G(s\xi)\in S^1$. For $X\in\frg$, we have
\begin{align*}
    z\exp_G(tX)=z\exp_G(tX)z^{-1}z = \exp_G(t\Ad_z(X))z.
\end{align*}
It therefore follows from (\ref{contactlabel30}) and Remark~\ref{contactlabel32} that
\begin{align*}
    &(\widetilde{\dolb_\Dcal}\alpha)_{\vert gz}(X^{1,0}_{l_1},\ldots,X^{1,0}_{l_p},X^{0,1}_{k_0},\ldots,X^{0,1}_{k_q})\\
    =& \begin{aligned}[t]
            \sum\limits_{j=0}^q (-1)^{p+j} \frac{\partial}{\partial t}_{\vert t=0} &(\alpha_{\vert gz\exp_G(tX_{k_j})} -i  \alpha_{\vert gz\exp_G(tJ(X_{k_j}))}) \\
    &(X^{1,0}_{l_1},\ldots,X^{1,0}_{l_p},X^{0,1}_{k_0},\ldots,\widehat{X^{0,1}_{k_j}},\ldots,X^{0,1}_{k_p})
    \end{aligned}\\
    =& \begin{aligned}[t]
        \sum\limits_{j=0}^q (-1)^{p+j} \frac{\partial}{\partial t}_{\vert t=0} &(\alpha_{\vert g\exp_G(t\Ad_z(X_{k_j})z)} -i  \alpha_{\vert g\exp_G(t\Ad_z(J(X_{k_j})))z}) \\
    &(X^{1,0}_{l_1},\ldots,X^{1,0}_{l_p},X^{0,1}_{k_0},\ldots,\widehat{X^{0,1}_{k_j}},\ldots,X^{0,1}_{k_p})
    \end{aligned}\\
    =& \sum\limits_{j=0}^q (-1)^{p+j} e^{-m\cdot is} \Lcal_{\Ad_z(X_{k_j}-iJ(X_{k_j}))}(\alpha(X^{1,0}_{l_1},\ldots,X^{1,0}_{l_p},X^{0,1}_{k_0},\ldots,\widehat{X^{0,1}_{k_j}},\ldots,X^{0,1}_{k_p}))_{\vert g}\\
    =& \sum\limits_{j=0}^q (-1)^{p+j} e^{-m\cdot is} \Lcal_{\Ad_z(X^{0,1}_{k_j})}(\alpha(X^{1,0}_{l_1},\ldots,X^{1,0}_{l_p},X^{0,1}_{k_0},\ldots,\widehat{X^{0,1}_{k_j}},\ldots,X^{0,1}_{k_p}))_{\vert g}\\
    =& e^{-(m+1)\cdot is} (\widetilde{\dolb_\Dcal}\alpha)_{\vert g} (X^{1,0}_{l_1},\ldots,X^{1,0}_{l_p},X^{0,1}_{k_0},\ldots,X^{0,1}_{k_q}).
\end{align*}
\noindent Thus, $\widetilde{\dolb_\Dcal}\colon C^\infty(G,\Lambda^{p,q} \frp^\ast)^H_m\rightarrow C^\infty(G,\Lambda^{p,q} \frp^\ast)^H_{m+1}$.\\
We have $\mu\circ \widetilde{\dolb^{\Lcal^{m+p-q}}} = \widetilde{\dolb_\Dcal}\circ \mu$, i.e.\ there is a commutative diagram
\begin{equation*}
    \begin{tikzcd}
        \frA^{p,q}(\Dcal)_m
        \arrow[r, "\psi^{-1}\circ\mu\circ\psi"]
        \arrow[d, "\dolb_\Dcal"']
        &
        \frA^{p,q}(G/K,\Lcal^{m+p-q})
        \arrow[d, "\dolb^{\Lcal^{m+p-q}}"]
        \\
        \frA^{p,q+1}(\Dcal)_{m+1}
        \arrow[r, "\psi^{-1}\circ\mu\circ\psi"']
        &
        \frA^{p,q+1}(G/K,\Lcal^{m+p-q}).
    \end{tikzcd}
\end{equation*}
In this regard, $\dolb^{\Lcal^{m+p-q}}$ is precisely the Dolbeault operator on $\Lcal^{m+p-q}$, which is defined for $\alpha\in C^\infty (G,L^{m+p-q})^K$ by
\begin{align*}
    (\widetilde{\dolb^{\Lcal^{m+p-q}}}\alpha)(X^{0,1}_j):= \Lcal_{X^{0,1}_j}(\alpha),
\end{align*}
see also \cite{tirao1970homogenous} and \cite{kobayashiDGMCVB}. This relationship between the operators mentioned above was already discussed in \cite{rumin2026}, but we want to derive it once again ourselves here by only using representation theory.\\
Since $\mu$ is an isometry with respect to the scalar products defined in (\ref{contactlabel33}), the same statement also follows for the adjoints of the operators.\\
For $m+p-q\in N\cdot \Z$, let
\begin{itemize}
    \item $\Hcal^{p,q}_{\dolb_\Dcal}(G/H):=\kernel(\Delta_{\dolb_\Dcal})\cap \frA^{p,q}(\Dcal)$ denote the space of $\dolb_\Dcal$-harmonic forms of degree $(p,q)$,
    \item $\Hcal^{p,q}_{\dolb_\Dcal,m}(G/H):=\kernel(\Delta_{\dolb_\Dcal})\cap \frA^{p,q}(\Dcal)_m$ denote the space of $\dolb_\Dcal$-harmonic forms of degree $(p,q)$ with weight $m$,
    \item $\Hcal^{p,q}_{\dolb}(G/K,\Lcal^{m+p-q}):=\kernel(\Delta_{\dolb^{\Lcal^{m+p-q}}})\cap \frA^{p,q}(G/K,\Lcal^{m+p-q})$ denote the space of $\dolb^{\Lcal^{m+p-q}}$-harmonic forms of degree $(p,q)$, where $\Delta_{\dolb^{\Lcal^{m+p-q}}} = \dolb^{\Lcal^{m+p-q}}(\dolb^{\Lcal^{m+p-q}})^\ast + (\dolb^{\Lcal^{m+p-q}})^\ast \dolb^{\Lcal^{m+p-q}}$.
\end{itemize}
According to \cite[Theorem~5.24]{Voisin}, the space $\Hcal^{p,q}_{\dolb}(G/K,\Lcal^{m+p-q})$ of $\dolb^{\Lcal^{m+p-q}}$-harmonic forms is isomorphic to the Dolbeault cohomology 
\begin{align*}
    H^{p,q}_{\dolb}(G/K,\Lcal^{m+p-q}) = \dfrac{
\kernel (\dolb^{\Lcal^{m+p-q}}_{\vert \frA^{p,q}(G/K,\Lcal^{m+p-q})})
}{
\image (\dolb^{\Lcal^{m+p-q}}_{\vert \frA^{p,q-1}(G/K,\Lcal^{m+p-q})})
}.
\end{align*}
From the results above, we directly obtain the following Lemma.
\begin{lemma}
    For $m+p-q\in N\cdot \Z$, we have
    \begin{align*}
        \Hcal^{p,q}_{\dolb_\Dcal,m}(G/H) \cong \Hcal^{p,q}_{\dolb}(G/K,\Lcal^{m+p-q}) \cong H^{p,q}_{\dolb}(G/K,\Lcal^{m+p-q}).
    \end{align*}
\end{lemma}
\noindent This yields
\begin{align}
    \dim (\Hcal_{\dolb_b,m}^{p,q}(G/H)) = \dim (H^{p,q}_{\dolb}(G/K,\Lcal^{m+p-q})),\label{contactlabel5}
\end{align}
which, together with the decomposition from (\ref{contactlabel34}) and the eigenvalues from (\ref{contactlabel6}), will be helpful in calculating the contact torsion.\\
At this point, we can also interpret the constant term $A$ in the zeta function very clearly. From (\ref{contactlabel12}) we know that
\begin{align*}
    A = \sum\limits_{k=0}^{2n} (-1)^k \dim_\C (\kernel(-\Lcal_T^2)\cap \kernel (\Delta_{\dolb_\Dcal})\cap\frA^k_\C(\Dcal)).
\end{align*}
The kernel of $-\Lcal_T^2$ is obtained directly from (\ref{contactlabel6}) by considering all $m\in\Z$ and $p,q\geq 0$ such that $m+p-q=0$. Due to equation (\ref{contactlabel5}) above and classical Hodge theory \cite[Section~0.7]{griffiths1994principles}, we immediately find that
\begin{align}
    A = \sum_{p,q\geq 0} (-1)^{p+q} \dim (H^{p,q}_{\dolb}(G/K)) = \chi(G/K). \label{contactlabel36}
\end{align}
\subsection{Computation of the Contact Torsion}
\label{contactsect8.3}
\begin{theorem}
\label{contactlabel37}
    Let $G/H$ a compact CR symmetric contact space as in Section~\ref{contactsect8.1}. We have
    \begin{align*}
        Z(s) = \chi(G/K)\left(1+ \left(2\frac{n}{N}\right)^{2s} 2\zeta_\mathrm{R} (2s)\right).
    \end{align*}
    Therefore, the analytic contact torsion is given by
    \begin{align*}
        T_\Ecal(G/H) = {\chi (G/K)} \cdot \ln\left(4\pi \frac{n}{N}\right) = {\frac{\vert W_G\vert}{\vert W_K\vert}} \cdot \ln\left(4\pi \frac{n}{N}\right),
    \end{align*}
    where $W_G$ and $W_K$ are the Weyl groups of $G$ and $K$, and we also have 
    \begin{align*}
        Z(0)=0.
    \end{align*}
\end{theorem}
\noindent As in \cite{rumin2012analytic} and \cite{kitaokaspheres}, $Z(0)=0$ implies that the torsion is invariant under constant scaling of $\theta$, i.e.\ $\theta\mapsto K\theta$ for a constant $K$.
\begin{proof}[{Proof of Theorem~\ref{contactlabel37}}]
    In (\ref{contactlabel5}), we observed that
    \begin{align*}
        \dim (\Hcal_{\dolb_b,m}^{p,q} (G/H)) = \dim (H^{p,q}_{\dolb}(G/K,\Lcal^{m+p-q}))
    \end{align*}
    and on this space the corresponding eigenvalue of $-\Lcal_T^2$ is precisely $\left(\frac{m+p-q}{2n}\right)^2$, as shown in (\ref{contactlabel6}).\\
    For the zeta function, (\ref{contactlabel35}) together with the results from \cite{rumin2026} then yields
    \begin{align*}
    Z(s) =& A + \sum\limits_{m\in\mathbb{Z}} \sum\limits_{\substack{p,q\geq 0\\ m+p-q\in N\cdot\Z\setminus \lbrace 0\rbrace}} (-1)^{p+q} \dim (H^{p,q}_{\dolb}(G/K,\Lcal^{m+p-q})) \frac{1}{\left(\frac{m+p-q}{2n}\right)^{2s}}\\
    =& A + (2n)^{2s} \sum\limits_{m\in\mathbb{Z}} \sum\limits_{\substack{p,q\geq 0\\ m+p-q\in N\cdot\Z\setminus \lbrace 0\rbrace}} (-1)^{p+q} \dim (H^{p,q}_{\dolb}(G/K,\Lcal^{m+p-q})) \frac{1}{(m+p-q)^{2s}}\\
    =& A + (2n)^{2s} \sum\limits_{j\in\mathbb{Z}\setminus \lbrace 0\rbrace} \sum_{p,q\geq 0} (-1)^{p+q} \dim (H^{p,q}_{\dolb}(G/K,\Lcal^{N\cdot j})) \frac{1}{(N\cdot j)^{2s}}\\
    =& A + \left(2\frac{n}{N}\right)^{2s} \sum\limits_{j\in\mathbb{Z}\setminus \lbrace 0\rbrace} \frac{1}{j^{2s}} \sum_{p,q\geq 0} (-1)^{p+q} \dim (H^{p,q}_{\dolb}(G/K,\Lcal^{N\cdot j}))\\
    =& A + \left(2\frac{n}{N}\right)^{2s} \sum\limits_{j\in\mathbb{Z}\setminus \lbrace 0\rbrace} \frac{1}{j^{2s}} \sum_{p\geq 0} (-1)^{p} \chi_{\operatorname{hol}} (G/K,\Lambda^{p,0}T^\ast (G/K)\otimes \Lcal^{N\cdot j})
    \end{align*}
    with the holomorphic Euler characteristic $\chi_{\operatorname{hol}} (G/K,\Lambda^{p,0}T^\ast (G/K)\otimes \Lcal^{N\cdot j})$. The Hirzebruch-Riemann-Roch theorem from \cite[Theorem~3]{AtiyahSingerIndexTheorem} allows us to study these terms with ease. The Hirzebruch-Riemann-Roch theorem, together with $\operatorname{ch} (E\otimes F) = \operatorname{ch} (E)\operatorname{ch} (F)$, states
\begin{align*}
    \chi_{\operatorname{hol}} (G/K,\Lambda^{p,0}T^\ast (G/K)\otimes \Lcal^{N\cdot j})=& \int_{G/K} \operatorname{ch} (\Lambda^{p,0}T^\ast (G/K)\otimes \Lcal^{N\cdot j}) \operatorname{Td} (G/K)\\
    =& \int_{G/K} \operatorname{ch} (\Lambda^{p,0}T^\ast (G/K)) \operatorname{ch} (\Lcal^{N\cdot j}) \operatorname{Td} (G/K),
\end{align*}
and thus
\begin{align*}
    &\sum_{p\geq 0} (-1)^{p} \chi_{\operatorname{hol}} (G/K,\Lambda^{p,0}T^\ast (G/K)\otimes \Lcal^{N\cdot j})\\
    =& \sum_{p\geq 0} (-1)^{p} \int_{G/K} \operatorname{ch} (\Lambda^{p,0}T^\ast (G/K)) \operatorname{ch} (\Lcal^{N\cdot j})\operatorname{Td} (G/K)\\
    =& \int_{G/K} (\sum_{p\geq 0} (-1)^{p}\operatorname{ch} (\Lambda^{p,0}T^\ast (G/K))) \operatorname{Td} (G/K) \operatorname{ch} (\Lcal^{N\cdot j})=: (\star).
\end{align*}
According to \cite[Example 3.2.5]{FultonIntersectionTheory}, we have
\begin{align*}
    \left(\sum_{p\geq 0} (-1)^{p}\operatorname{ch} (\Lambda^{p,0}T^\ast (G/K))\right)\operatorname{Td} (G/K) = c_n(T^{1,0} (G/K)),
\end{align*}
which is a $2n$-form. Therefore, only the term in degree $0$ in the Chern character $\operatorname{ch} (\Lcal^{N\cdot j}) = \operatorname{rk} (\Lcal^{N\cdot j}) + (\text{higher degrees}) = 1 + (\text{higher degrees})$ is relevant for the calculation of the integral, so it follows, together with the Chern-Gauß-Bonnet theorem \cite{griffiths1994principles}, that
\begin{align*}
    (\star) =& \int_{G/K} c_n(T^{1,0} (G/K)) = \chi(G/K)
\end{align*}
with the Euler characteristic $\chi(G/K)$.\\
Similarly, if we consider the trivial complex line bundle instead of $\Lcal^{N\cdot j}$, we obtain $A=\chi(G/K)$, as we have already seen in (\ref{contactlabel36}).\\
Overall, we therefore obtain
\begin{align*}
    Z(s) =& \chi(G/K) + \left(2\frac{n}{N}\right)^{2s} \sum\limits_{j\in\mathbb{Z}\setminus \lbrace 0\rbrace} \frac{1}{j^{2s}} \chi(G/K)\\
    =& \chi(G/K)+\chi(G/K) \left(2\frac{n}{N}\right)^{2s} \sum\limits_{j\in\mathbb{Z}\setminus \lbrace 0\rbrace} \frac{1}{j^{2s}}\\
    =& \chi(G/K)\left(1+ \left(2\frac{n}{N}\right)^{2s} 2\zeta_\mathrm{R} (2s)\right)
\end{align*}
for the Riemann zeta function $\zeta_\mathrm{R}$. In \cite{WangHomogeneousEulerChar}, it was proven that $\chi(G/K) = \frac{\vert W_G\vert}{\vert W_K\vert} > 0$ for the Weyl groups $W_G$ and $W_K$. Since
\begin{align*}
    \zeta_\mathrm{R}(0) =& -\frac{1}{2},\\
    \frac{\partial}{\partial s}_{\vert s=0} \zeta_\mathrm{R}(2s) =& 2\zeta_\mathrm{R}'(0) = 2\left(-\frac{1}{2}\ln (2\pi)\right)=-\ln (2\pi),\\
    \frac{\partial}{\partial s}_{\vert s=0} \left(2\frac{n}{N}\right)^{2s} =& 2\ln \left(2\frac{n}{N}\right),
\end{align*}
as shown in \cite[Chapter 13]{WhittakerWatson}, we have
\begin{align*}
    Z(0)=&0,\\
    Z'(0) =& 2\chi(G/K) \left(2\ln\left(2\frac{n}{N}\right)\cdot \left(-\frac{1}{2}\right) -\ln (2\pi)\right)\\
    =& -2\chi(G/K) \left(\ln \left(2\frac{n}{N}\right)+\ln (2\pi)\right)= -2\chi(G/K) \ln\left(4\pi\frac{n}{N}\right),
\end{align*}
and for the contact torsion $T_\Ecal(G/H)$ we conclude
\begin{align*}
    T_\Ecal(G/H) =& -\frac{1}{2}Z'(0) = {\chi(G/K)} \cdot \ln\left(4\pi\frac{n}{N}\right)\\
    =& {\frac{\vert W_G\vert}{\vert W_K\vert}}\cdot \ln\left(4\pi\frac{n}{N}\right).
\end{align*}
\end{proof}
\begin{remark}
    As we shall see in the next Example~\ref{contactlabel39}, $N$ is not always equal to $n$.
\end{remark}
Let the base manifold $G/K = G_1/K_1 \times G_2/K_2$ be given as a product of spaces as in Section~\ref{contactsect8.1}, where $G_1$ and $G_2$ are not necessarily simple Lie groups, but may also be semisimple, so that $G_1/K_1$ and $G_2/K_2$ can be further decomposed into products. This is always possible, since $G/K$ is simply connected.\\
As before, we also assume that there exist Lie subgroups $H\subset K$ and $H_j\subset K_j$ such that $G/H$ and $G_j/H_j$ are compact CR symmetric contact manifolds. Let $2n:=\dim(G/K)$ and $2n_j:=\dim (G_j/K_j)$, i.e.\ in particular $n=n_1+n_2$. We denote by $N$ and $N_j$ the orders of $\Gamma$ and $\Gamma_j$ respectively.\\
In \cite[Theorem~3.3]{RaySingerTorsionComplexManifolds} and \cite[Lemma 2]{kkhermsymm}, one can find product formulas for the holomorphic torsion of a product $M_1\times M_2$ of two compact K\"ahler manifolds $M_1$ and $M_2$. Although the product of two contact manifolds is no longer a contact manifold, we can nevertheless obtain a kind of product formula by using the decomposition of $G/K$, so that we can express the contact torsion $T_\Ecal(G/H)$ in terms of the contact torsions $T_\Ecal(G_j/H_j)$.
\begin{lemma}
    The contact torsion of $G/H$ is given by
    \begin{align*}
        T_\Ecal(G/H) =& \frac{1}{2}\left( \chi(G_2/K_2)T_\Ecal(G_1/H_1) + \chi(G_1/K_1)T_\Ecal(G_2/H_2) \right)\\
        &+ \frac{\chi(G/K)}{2}\left( \ln\left(\frac{N_1\cdot N_2}{N^2}\right) + 2\ln\left(\sqrt{\frac{n_1}{n_2}}+\sqrt{\frac{n_2}{n_1}}\right)\right).
    \end{align*}
\end{lemma}
\begin{proof}
    According to Theorem~\ref{contactlabel37}, we have
    \begin{align*}
        T_\Ecal(G/H) =&  {\chi (G/K)}\cdot \ln\left(4\pi \frac{n}{N}\right) = {\chi (G/K)}\cdot \ln\left(4\pi \frac{n_1+n_2}{N}\right),\\
        T_\Ecal(G_j/H_j) =&  {\chi (G_j/K_j)}\cdot \ln\left(4\pi \frac{n_j}{N_j}\right).
    \end{align*}
    Since $a+b = \sqrt{a\cdot b}\left(\sqrt{\frac{a}{b}}+\sqrt{\frac{b}{a}}\right)$ for $a,b>0$, it follows that
    \begin{align*}
        \ln\left(4\pi \frac{n_1+n_2}{N}\right) = \frac{1}{2}\ln \left(4\pi \frac{n_1}{N}\right) + \frac{1}{2}\ln\left(4\pi \frac{n_2}{N}\right)+\ln\left(\sqrt{\frac{n_1}{n_2}} + \sqrt{\frac{n_2}{n_1}}\right)
    \end{align*}
    for $a=4\pi \frac{n_1}{N}$, $b=4\pi \frac{n_2}{N}$. Since
    \begin{align*}
        \ln \left(4\pi \frac{n_j}{N}\right) = \ln \left(4\pi \frac{n_j}{N_j}\cdot \frac{N_j}{N}\right) = \ln \left(4\pi \frac{n_j}{N_j}\right) + \ln \left(\frac{N_j}{N}\right)
    \end{align*}
    and $\ln \left(\frac{N_1}{N}\right)+\ln \left(\frac{N_2}{N}\right) = \ln \left(\frac{N_1\cdot N_2}{N^2}\right)$, we obtain the desired formula by also using the following formula for the Euler characteristic
    \begin{align*}
        \chi(G/K) = \chi (G_1/K_1\times G_2/K_2) = \chi(G_1/K_1)\cdot\chi(G_2/K_2).
    \end{align*}

\end{proof}
\begin{example}
\label{contactlabel39}
    \begin{enumerate}
        \item For the sphere 
        \begin{align*}
            G/H = S^{2n+1} = \dfrac{\SU(n+1)}{\SU (n)},
        \end{align*}
        we have
        \begin{align*}
            G/K = \mathbf{P}^n\C = \dfrac{\SU(n+1)}{\operatorname{S}(\operatorname{U} (n)\times \operatorname{U} (1))},
        \end{align*}
        where
        \begin{align*}
            K = \operatorname{S}(\operatorname{U} (n)\times \operatorname{U} (1)) = \left\{ \begin{pmatrix}
                A&0\\0&z
            \end{pmatrix} \mid A\in \operatorname{U} (n), z\in \operatorname{U} (1)=S^1, \det(A)z=1     
            \right\}.
        \end{align*}
        Thus, the order of $\Gamma$ is $N=n$ and we have $\chi(\mathbf {P}^n\C)=n+1$. Consequently, the torsion is $T_\Ecal(S^{2n+1})={(n+1)}\cdot \ln(4\pi)$, which Kitaoka had already determined in \cite[Theorem~1.2]{kitaokaspheres}.
        \item For
        \begin{align*}
            G/H = \dfrac{\spn{n}}{\SU (n)},
        \end{align*}
        we have
        \begin{align*}
            G/K = \dfrac{\spn{n}}{\operatorname{U} (n)}.
        \end{align*}
        Similar to the previous example, the order of $\Gamma$ is $n$ and $\dim (G/H) = (2n^2+n)-(n^2-1)=n^2+n+1=n(n+1)+1$. Due to $\vert W(\spn{n})\vert = 2^n\cdot n!$ and $\vert W(\operatorname{U}(n))\vert = n!$, we obtain
        \begin{align*}
            T_\Ecal\left(\dfrac{\spn{n}}{\SU (n)}\right)={2^n}\cdot \ln\left(4\pi \frac{n(n+1)}{2n}\right)={2^n}\cdot \ln\left(2\pi (n+1)\right).
        \end{align*}
        The Weyl group and further information on these Lie groups can be found in \cite[Section~V.6]{broetomdieck} and \cite[Chapter~6]{sepanski}.
    \end{enumerate}
    
\end{example}
\subsection{Equivariant Contact Torsion}
\label{contactsect8.4}
Let $G/H$ again be a compact CR symmetric contact space as in Section~\ref{contactsect8.1}.\\
Since $G$ acts on $G/H$ via contactomorphisms that preserve the contact form and metric, we can also consider the equivariant contact torsion for any $g \in G$. First, let us recall the definition of the equivariant contact torsion, see \cite{tessmer} for further details.\\

In \cite{tessmer}, it was already shown that the action of $g\in G$, by pullback on differential forms, yields an action on the spaces of the Rumin complex and the action commutes with the operators that appear in the Rumin complex, as well as with their adjoints and the Hodge $\ast$ operator, thus,
\begin{align*}
    [g,d_\Ecal]=[g,d_\Ecal^\ast]=[g,D]=[g,D^\ast]=[g,\ast]=0.
\end{align*}
Therefore,
\begin{align*}
    [g,\Delta_\Ecal]=0,
\end{align*}
which is why the action of $g\in G$ preserves the eigenspaces of $\Delta_\Ecal$.\\
Since the contact form and the complex structure are $G$-invariant, the action of $g\in G$ also commutes with $L, d_\Dcal, \dolb_\Dcal$ and $\partial_\Dcal$, i.e.
\begin{align*}
    [g,d_\Dcal]=[g,\partial_\Dcal]=[g,\dolb_\Dcal]=[g,L]=0.
\end{align*}
For $g\in G$, the equivariant zeta function is defined for $\operatorname{Re} (s)\gg 1$ by
\begin{align*}
    \zeta(\Delta_\Ecal\vert \Ecal^k(M),g) := \Tr\left(g_{\vert H_\Ecal^k(M)}\right) + \sum\limits_{\lambda\in\operatorname{spec}^\ast (\Delta_\Ecal\vert \Ecal^k(M))} \Tr\left(g_{\vert \operatorname{Eig}_\lambda(\Delta_\Ecal\vert\Ecal^k(M))}\right) \frac{1}{\lambda^s},
\end{align*}
where the action of $g\in G$ is defined by $g\cdot \alpha = (L_{g^{-1}})^\ast\alpha$ for $\alpha\in \Ecal^k(M)$, such that the action on the differential forms remains consistent with the action used previously.\\
As in the non-equivariant case, the zeta function can be extended to a meromorphic function on $\C$ with the help of the Mellin transform, which is holomorphic in $s=0$, see \cite[Chapter~5]{tessmer}. We have
\begin{align*}
    \zeta(\Delta_\Ecal\vert \Ecal^k(M),g)(s) := \Tr\left(g_{\vert H_\Ecal^k(M)}\right) + \frac{1}{\Gamma (s)}\int\limits_0^{+\infty} \Tr^\ast(g\circ e^{-t\Delta_\Ecal}\vert \Ecal^k(M)) t^{s-1}dt,
\end{align*}
The \textbf{equivariant contact zeta function} of the Rumin complex if then defined as
\begin{align*}
    Z_g(s) := \frac{1}{2} \sum\limits_{k=0}^{2n+1} (-1)^k \omega(k) \zeta(\Delta_\Ecal\vert \Ecal^k(M),g)(s),
\end{align*}
where 
\begin{align*}
    \omega(k):=\begin{cases}
    k &\text{ for } k\leq n,\\
    k+1 &\text{ for } k\geq n+1.
\end{cases}
\end{align*}
As in the non-equivariant case, the zeta function can also be expressed as follows
\begin{align*}
    Z_g(s)=\sum\limits_{k=0}^n (-1)^k (n+1-k) \zeta(\Delta_\Ecal\vert \Ecal^k(M),g)(s).
\end{align*}
The \textbf{analytic equivariant (contact) torsion} of the Rumin complex is defined by
\begin{align*}
    T_\Ecal (M,g) := {-\frac{1}{2}Z'_g(0)}.
\end{align*}
For $g=e$, we have $T_\Ecal (M,e)=T_\Ecal (M)$.
\begin{remark}
    As with non-equivariant torsion, we omit the exponential function here, contrary to Rumin's original definition.
\end{remark}
The result from \cite{rumin2026}, which we summarised in Section~\ref{contactsect7.2}, can be extended to the equivariant case due the $G$-invariance of the relevant operators, by considering
\begin{align*}
    \vartheta_g (t) := \sum\limits_{k=0}^n (-1)^k (n+1-k)\Tr(g\circ e^{-t\Delta_\Ecal}\vert \Ecal^k(M))
\end{align*}
instead of
\begin{align*}
    \vartheta (t) := \sum\limits_{k=0}^n (-1)^k (n+1-k)\Tr(e^{-t\Delta_\Ecal}\vert \Ecal^k(M)).
\end{align*}
Analogous to (\ref{contactlabel35}), we obtain
\begin{align*}
    Z_g(s) =& \sum\limits_{k=0}^n (-1)^k(n+1-k)\zeta(\Delta_\Ecal\vert \Ecal^k(M),g)(s) \\
    =& \sum\limits_{k=0}^{2n} (-1)^k \zeta (-\Lcal_T^2\vert \kernel (\Delta_{\dolb_\Dcal})\cap\frA^k_\C(\Dcal),g)(s),
\end{align*}
where 
\begin{align*}
    &\zeta (-\Lcal_T^2\vert \kernel (\Delta_{\dolb_\Dcal})\cap\frA^k_\C(\Dcal),g)(s) \\
    := &\Tr \left(g_{\vert \kernel(-\Lcal_T^2)\cap \kernel (\Delta_{\dolb_\Dcal})\cap\frA^k_\C(\Dcal)}\right)\\
    &+ \sum\limits_{\lambda\in\operatorname{spec}^\ast (-\Lcal_T^2\vert \kernel (\Delta_{\dolb_\Dcal})\cap\frA^k_\C(\Dcal))} \Tr\left(g_{\vert \operatorname{Eig}_\lambda(-\Lcal_T^2\vert\kernel (\Delta_{\dolb_\Dcal})\cap\frA^k_\C(\Dcal))}\right)\frac{1}{\lambda^s}.
\end{align*}
For $\alpha\in C^\infty(G,\Lambda^{p,q}\frp^\ast)^H_m$ and $x\in G$, we have
\begin{align*}
    (g\cdot \alpha)_{\vert x\exp_G(t\xi)}=\alpha_{\vert g^{-1}x\exp_G(t\xi)}=e^{-m\cdot it}\alpha_{\vert g^{-1}x},
\end{align*}
thus, the action of $g\in G$ maps $\Hcal_{\dolb_b,m}^{p,q} (G/H)$ onto itself.\\
The isomorphism of $\Hcal_{\dolb_b,m}^{p,q} (G/H)$ and $H^{p,q}_{\dolb}(G/K,\Lcal^{m+p-q})$ is $G$-equivariant, since the complex structure and metric on $G/K$ and $\Lcal^{m+p-q}=G\times_K L^{m+p-q}$  are also $G$-invariant. Analogously to Theorem~\ref{contactlabel37}, we obtain the following lemma:
\begin{lemma}
    For a compact CR symmetric contact space as in Section~\ref{contactsect8.1}, the equivariant zeta function with respect to the action of $g\in G$ is given by
    \begin{align}
        Z_g(s) =& \sum_{p,q\geq 0} (-1)^{p+q} \Tr \left( g_{\vert H_{\dolb}^{p,q}(G/K)} \right) \notag\\
        &+ \left(2\frac{n}{N}\right)^{2s} \sum\limits_{j\in\mathbb{Z}\setminus \lbrace 0\rbrace} \frac{1}{j^{2s}} \sum_{p,q\geq 0} (-1)^{p+q} \Tr \left( g_{\vert H^{p,q}_{\dolb}(G/K,\Lcal^{N\cdot j})} \right). \label{contactlabel40}
    \end{align}
\end{lemma}
The action of $g\in G$ on $\Lcal^{m+p-q}$ is defined by
\begin{align*}
    g\cdot [x,v]=[gx,v].
\end{align*}
We would now like to examine the term
\begin{align*}
    \sum_{p,q\geq 0} (-1)^{p+q} \Tr \left( g_{\vert H^{p,q}_{\dolb}(G/K,\Lcal^{N\cdot j})} \right)
\end{align*}
from (\ref{contactlabel40}) in more detail. For some arbitrary $g\in G$, an approach involving the Atiyah-Segal-Singer index formula \cite[Theorem~13]{kkhermsymm}
\begin{align*}
    \sum\limits_{q\geq 0} (-1)^{q} \Tr \left( g_{\vert H^{p,q}_{\dolb}(G/K,\Lcal^{N\cdot j})} \right) = \int_{(G/K)_g} \operatorname{Td}_g (G/K)\operatorname{ch}_g (\Lambda^{p,0}T^\ast (G/K)\otimes \Lcal^{N\cdot j})
\end{align*}
might be possible, where $(G/K)_g$ denotes the fixed point submanifold and $\operatorname{Td}_g$, $\operatorname{ch}_g$ are the equivariant versions of $\operatorname{Td}$, $\operatorname{ch}$. However, in this case there is no such simplification of the equivariant characteristic classes as before, so that this term continues to depend on $j$ and does not simplify as drastically as in the non-equivariant case.\\
For elements $t=\exp_G(X)\in T\subset K\subset G$ where $X\in\frt$ of the maximal torus, which generate the torus, another possible approach would be to use the Atiyah-Bott fixed point theorem for the calculation. Here, $\frt\subset\frk$ is the Lie algebra of $T$.
\begin{theorem}[{\cite[Theorem~4.12]{atiyahbott2}}, Atiyah-Bott fixed point theorem for complex manifolds]
    Let $N$ be a compact complex manifold with $\dim_\R (N)=2n$ and $F\rightarrow N$ a holomorphic vector bundle. For a holomorphic map $\gamma\colon N\rightarrow N$ with $\det(1-T_x\gamma)\neq 0$ for all $x\in N^\gamma$ and a holomorphic vector bundle homomorphism $\psi^F \colon \gamma^\ast F\rightarrow F$, which provide an endomorphism of the Dolbeault complex
    \begin{align*}
        0\rightarrow\frA^{p,0}(N,F)\xrightarrow{\dolb^F}\frA^{p,1}(N,F)\xrightarrow{\dolb^F}\ldots \xrightarrow{\dolb^F}\frA^{p,n}(N,F)\rightarrow 0,
    \end{align*}
    i.e.\ $\psi_q=\gamma^\ast\otimes \psi^F$ for the pullback of forms $\gamma^\ast\colon \gamma^\ast(\Lambda^{p,q}T^\ast N)\rightarrow \Lambda^{p,q}T^\ast N$, we have
    \begin{align*}
        \sum\limits_{q=0}^n (-1)^q \Tr \left(\tau_{\vert H_{\dolb}^{p,q}(N,F)}\right)
        = \sum\limits_{x\in N^\gamma} \frac{\Tr (\gamma^\ast_{\vert \Lambda^{p,0}T^\ast N,x})\cdot \Tr(\psi^F_{\vert x})}{ \det\nolimits_\C (1-T^{1,0}_x \gamma)}.
    \end{align*}
\end{theorem}
Let $g=t\in T$ an element of the maximal torus $T$, which generates $T$, i.e.\ $\overline{\langle t\rangle_\Z}=T$. In our case, $\gamma = L_{t^{-1}}$ and $\psi^F$ is defined by $(gK,[t^{-1}g,v])\mapsto [g,v]$ with $F=\Lcal^{N\cdot j}=G\times_K L^{N\cdot j}$. The fixed points of the holomorphic map $L_{t^{-1}}\colon G/K\rightarrow G/K$ are always non-degenerate on $G/K$ and are given by $n_1K,\ldots,n_k K$ for $W_G/W_K=\lbrace (n_1T)W_K,\ldots,(n_kT)W_K\rbrace$ and $k=\vert W_G/W_K\vert$. Therefore, we obtain
\begin{align*}
    &\sum\limits_{p,q\geq 0} (-1)^{p+q} \Tr \left( t_{\vert H^{p,q}_{\dolb}(G/K,\Lcal^{N\cdot j})} \right) = \sum\limits_{p=0}^n (-1)^p\sum\limits_{q=0}^n (-1)^q \Tr \left( t_{\vert H^{p,q}_{\dolb}(G/K,\Lcal^{N\cdot j})} \right)\\
    =& \sum\limits_{p=0}^n (-1)^p \sum\limits_{\substack{[w] = (n_jT)W_K\\ \in W_G/W_K}} \frac{\Tr ({L_{t^{-1}}^\ast}_{\vert \Lambda^{p,0}T^\ast G/K,n_jK})\cdot \Tr(t_{\vert \Lcal^{N\cdot j},n_jK})}{ \det\nolimits_\C (1-T^{1,0}_{n_jK} L_{t^{-1}})} \\
    =& \sum\limits_{\substack{[w] = (n_jT)W_K\\ \in W_G/W_K}} \frac{\left(\sum\limits_{p=0}^n (-1)^p\Tr ({L_{t^{-1}}^\ast}_{\vert \Lambda^{p,0}T^\ast G/K,n_jK})\right)\cdot \Tr(t_{\vert \Lcal^{N\cdot j},n_jK})}{ \det\nolimits_\C (1-T^{1,0}_{n_jK} L_{t^{-1}})} \\
    =& \sum\limits_{\substack{[w] = (n_jT)W_K\\ \in W_G/W_K}} \frac{\det\nolimits_\C (1-(T^{1,0}_{n_jK} L_{t^{-1}})^\ast_{\vert (T^{1,0} G/K)^\ast})\cdot \Tr(t_{\vert \Lcal^{N\cdot j},n_jK})}{ \det\nolimits_\C (1-T^{1,0}_{n_jK} L_{t^{-1}})}\\
    =& \sum\limits_{\substack{[w] = (n_jT)W_K\\ \in W_G/W_K}} \frac{\det\nolimits_\C (1-(T^{1,0}_{n_jK} L_{t^{-1}})^\ast_{\vert (T^{1,0} G/K)^\ast})}{ \det\nolimits_\C (1-T^{1,0}_{n_jK} L_{t^{-1}})}\cdot \Tr(t_{\vert \Lcal^{N\cdot j},n_jK})\\
    =& \sum\limits_{\substack{[w] = (n_jT)W_K\\ \in W_G/W_K}} \Tr(t_{\vert \Lcal^{N\cdot j},n_jK}) = \sum\limits_{[w] \in W_G/W_K} e^{2\pi i \cdot j(w\cdot \lambda) (X)},
\end{align*}
where $\lambda$ is the weight of $L^N$ and $e^{2\pi i \cdot j(w\cdot \lambda) (X)}\neq 1$, since $t$ generates the maximal torus. Otherwise, $\lambda=0$, which is obviously not the case, since $L^N$ is a non-trivial representation.\\
Analogously, we obtain
\begin{align*}
    \sum_{p,q\geq 0} (-1)^{p+q} \Tr \left( t_{\vert H^{p,q}(G/K)} \right) = \frac{\vert W_G\vert}{\vert W_K\vert} = \chi(G/K).
\end{align*}
This yields
\begin{align*}
    Z_t(s) =& \frac{\vert W_G\vert}{\vert W_K\vert} + \left(2\frac{n}{N}\right)^{2s} \sum\limits_{j\in\mathbb{Z}\setminus \lbrace 0\rbrace} \frac{1}{j^{2s}} \sum\limits_{[w] \in W_G/W_K} e^{2\pi i \cdot j(w\cdot \lambda) (X)}\\
    =& \frac{\vert W_G\vert}{\vert W_K\vert} + \left(2\frac{n}{N}\right)^{2s}\sum\limits_{[w] \in W_G/W_K} \sum\limits_{j\in\mathbb{Z}\setminus \lbrace 0\rbrace} \frac{e^{2\pi i \cdot j(w\cdot \lambda) (X)}}{j^{2s}}\\
    =& \frac{\vert W_G\vert}{\vert W_K\vert} + \left(2\frac{n}{N}\right)^{2s}\sum\limits_{[w] \in W_G/W_K} \left( \zeta_L(2s,2\pi (w\cdot \lambda) (X)) + \zeta_L(2s,-2\pi (w\cdot \lambda) (X)) \right),
\end{align*}
where $\zeta_L(s,2\pi a)=\sum\limits_{j\in \N} \frac{e^{2\pi i\cdot j\cdot a}}{j^s}$ denotes the Lerch zeta function for $\operatorname{Re}(s)>2$ and $a\in \R$, see \cite[Equation (41)]{kkhermsymm}. We shall examine the zeta function $Z_t(s)$ below, so that we ultimately arrive at the following theorem.
\begin{theorem}
\label{contactlabel41}
    Let $G/H$ a compact CR symmetric contact space as in Section~\ref{contactsect8.1} and $t=\exp_G(X)\in T$ a generating element of the maximal torus, i.e.\ $\overline{\langle t\rangle_\Z}=T$.\\
    The equivariant contact torsion is given by
    \begin{align*}
        T_\Ecal (M,t) =& \frac{\vert W_G\vert}{\vert W_K\vert}\left(\ln \left(4\pi\frac{n}{N}\right)+\gamma\right) + \sum\limits_{[w]\in W_G/W_K} \frac{1}{2}(\psi([(w\cdot \lambda)(X)]) + \psi([-(w\cdot \lambda)(X)]))\\
        =& T_\Ecal (M) + \frac{\vert W_G\vert}{\vert W_K\vert}\gamma + \sum\limits_{[w]\in W_G/W_K} \frac{1}{2}(\psi([(w\cdot \lambda)(X)]) + \psi([-(w\cdot \lambda)(X)])),
    \end{align*}
    where $\psi(x)=\frac{\Gamma'(x)}{\Gamma(x)}$ is the Digamma function, $\gamma$ is Euler's constant and $[(w\cdot \lambda)(X)]$, $[-(w\cdot \lambda)(X)]$ are the projections to $]0,1[$ for $(w\cdot \lambda)(X)\in\R\setminus\Z$, i.e.\ the unique representatives in $]0,1[$ modulo $1$. We also obtain
    \begin{align*}
        Z_t(0)=0.
    \end{align*}
\end{theorem}
\noindent The theorem above also extends to all $g\in G$ that are conjugate to such $t\in T$ with $\overline{\langle t\rangle_\Z}=T$.\\

According to \cite{kkhermsymm}, analytic continuation yields $\zeta_L(0,2\pi a)=\frac{1}{e^{-2\pi i\cdot a}-1}$ for $a\not\in \R\setminus\Z$. Thus,
\begin{align}
    &\zeta_L(0,2\pi a)+\zeta_L(0,-2\pi a)= \frac{1}{e^{-2\pi i\cdot a}-1} + \frac{1}{e^{2\pi i\cdot a}-1} \nonumber\\
    =& \frac{e^{2\pi i\cdot a}-1+e^{-2\pi i\cdot a}-1}{(e^{-2\pi i\cdot a}-1)(e^{2\pi i\cdot a}-1)}
    = \frac{2(\cos (2\pi a)-1)}{2(1-\cos (2\pi a))} = -1. \label{contactlabel44}
\end{align}
We have
\begin{align*}
    \frac{\partial}{\partial s} (\zeta_L(2s,2\pi a) + \zeta_L(2s,-2\pi a)) =& - \sum\limits_{j\in \Z\setminus \lbrace 0\rbrace} \frac{e^{2\pi i\cdot j\cdot a}\ln(j^2)}{j^{2s}}.
\end{align*}
This function is precisely the kind of function that was studied by K\"ohler in \cite[pp.~65--66]{kkReidemeisterTorsion}. According to \cite[Lemma 13]{kkReidemeisterTorsion}, its value at $s=0$ is given by
\begin{align}
    \frac{\partial}{\partial s}_{\vert s=0} (\zeta_L(2s,2\pi a) + \zeta_L(2s,-2\pi a)) = -2(\ln (2\pi) + \gamma) - (\psi([a])+\psi([-a])), \label{contactlabel45}
\end{align}
where $\psi(x)=\frac{\Gamma'(x)}{\Gamma(x)}$ is the Digamma function, $\gamma$ is Euler's constant and $[a],[-a]$ are the projections to $]0,1[$ for $a\in\R\setminus\Z$, i.e.\ the unique representatives in $]0,1[$ modulo $1$.

We can now proceed to prove Theorem~\ref{contactlabel41}.
\begin{proof}[{Proof of Theorem~\ref{contactlabel41}}]
    From (\ref{contactlabel44}), we obtain
    \begin{align*}
        \left(\sum\limits_{[w] \in W_G/W_K} \left( \zeta_L(2s,2\pi (w\cdot \lambda) (X)) + \zeta_L(2s,-2\pi (w\cdot \lambda) (X)) \right)\right)_{\vert s=0} = - \frac{\vert W_G\vert}{\vert W_K\vert}.
    \end{align*}
    Furthermore, (\ref{contactlabel45}) yields
    \begin{align*}
        &\frac{\partial}{\partial s}_{\vert s=0} \sum\limits_{[w] \in W_G/W_K} \left( \zeta_L(2s,2\pi (w\cdot \lambda) (X)) + \zeta_L(2s,-2\pi (w\cdot \lambda) (X)) \right) \\
        =& -2\cdot\frac{\vert W_G\vert}{\vert W_K\vert} (\ln(2\pi) + \gamma) - \sum\limits_{[w]\in W_G/W_K} (\psi ([(w\cdot \lambda) (X)]) + \psi ([-(w\cdot \lambda) (X)])).
    \end{align*}
    Thus,
    \begin{align*}
        Z_t(0) = \frac{\vert W_G\vert}{\vert W_K\vert} - \frac{\vert W_G\vert}{\vert W_K\vert} = 0.
    \end{align*}
    Since $\frac{\partial}{\partial s}_{\vert s=0} \left(2\frac{n}{N}\right)^{2s} = 2\ln \left(2\frac{n}{N}\right)$, we have
    \begin{align*}
        T_\Ecal(M,t) =&  \frac{\vert W_G\vert}{\vert W_K\vert}\cdot\ln\left(2\frac{n}{N}\right) + \frac{\vert W_G\vert}{\vert W_K\vert} \cdot(\gamma + \ln(2\pi)) \\
        &+ \sum\limits_{[w]\in W_G/W_K} \frac{1}{2}(\psi ([(w\cdot \lambda) (X)]) + \psi ([-(w\cdot \lambda) (X)])) \\
        =& \frac{\vert W_G\vert}{\vert W_K\vert}\cdot \left(\ln\left(4\pi\frac{n}{N}\right)+\gamma\right) \\
        &+  \sum\limits_{[w]\in W_G/W_K} \frac{1}{2}(\psi ([(w\cdot \lambda) (X)]) + \psi ([-(w\cdot \lambda) (X)])).
    \end{align*}
\end{proof}
\printbibliography[title={References}]
\end{document}